\documentclass [leqno, spanish, 10pt]{amsart}

\AtBeginDocument{%
   \def\MR#1{}
}

\AtBeginDocument{%
   \def\DOI#1{}
}
\AtBeginDocument{%
   \def\URL#1{}
   \def\ISBN#1{}
}

\usepackage[T1]{fontenc}
\usepackage[utf8]{inputenc}
\usepackage{geometry}

\usepackage{microtype}

\usepackage[headings, myheadings]{fullpage}

\usepackage{lmodern}
\usepackage{microtype}

\usepackage[all]{xy}
\usepackage{amssymb, amsmath, amsfonts, amsthm, amsbsy, amscd, amsxtra, stmaryrd}
\usepackage{mathabx}
\usepackage[bbgreekl]{mathbbol} 
\usepackage{pgf,interval}
\usepackage[mathscr]{euscript}
\usepackage{upgreek}
\usepackage{mathtools}
\usepackage{tikz}
\usepackage[inline]{enumitem}

\usetikzlibrary{arrows}
\usepackage{subcaption}

\usepackage{url}
\usepackage{hyperref}
\hypersetup{
	colorlinks=true, linktocpage=true, hidelinks
}

\usepackage{etoolbox}
\apptocmd{\sloppy}{\hbadness 10000\relax}{}{}

\usepackage{thmtools} 

\let\oldqedbox\qedsymbol
\newcommand{\twoqedbox}{\oldqedbox}

\declaretheorem[
name = Theorem,
style = plain,
numberwithin = section,
]{theorem}

\declaretheorem[
name = Lemma,
style = plain,
numberwithin = section,
]{lemma}

\declaretheorem[
name = Corollary,
style = plain,
numberwithin = section,
]{corollary}

\declaretheorem[
name = Definition,
style = plain,
numberwithin = section,
]{definition}

\declaretheoremstyle[
  style=examplestyle,
  headfont=\normalfont\bfseries,
  sharenumber = theorem,
  bodyfont=\normalfont,
     qed = {\hbox{$\triangleleft$}}
]{examplestyle}

\declaretheorem[
  style=examplestyle,
  title=Example,
  refname={example,examples},
  Refname={Example,Examples},
]{example}

\declaretheorem[
  style=examplestyle,
  title=Remark,
  refname={remark, remarks},
  Refname={Remark, Remarks},
]{remark}

\numberwithin{equation}{section}

\makeatletter
\def\l@subsection{\@tocline{2}{0pt}{1pc}{4.6em}{}}
\renewcommand{\tocsubsection}[3]{%
  \indentlabel{\@ifnotempty{#2}{\hspace*{2.3em}\makebox[2.3em][l]{%
    \ignorespaces#1 #2.\hfill}}}#3}
\makeatother

\newcommand{\at}{\mathsf{t}}

\newcommand{\amlt}{\circ^{t}}
\newcommand{\nucleus}{\mathscr{N}}
\newcommand{\nt}{\mathsf{t}}
\newcommand{\ass}{\mathscr{A}}

\newcommand{\pmlt}{\diamond}
\renewcommand{\sprod}{\odot}

\newcommand{\cc}{\omega}
\newcommand{\one}{\mathbb{1}}

\newcommand{\bw}{\mathscr{W}}
\newcommand{\bwu}{\mathscr{U}}
\newcommand{\bwl}{\mathscr{L}}
\newcommand{\acl}{\mathsf{r}}
\newcommand{\mcl}{\mathsf{m}}

\newcommand{\acr}{\bar{\mathsf{r}}}

\newcommand{\eperp}{\lb e \ra^{\perp}}

\newcommand{\herm}{\mathbb{Herm}}

\newcommand{\mat}{\mathbb{M}}
\newcommand{\sect}{\mathscr{K}}

\newcommand{\szero}{\mathbb{Nil}_{2}}

\newcommand{\idem}{\mathbb{Idem}}
\newcommand{\spec}{\operatorname{Spec}}
\newcommand{\specp}{\spec^{\perp}}

\newcommand{\suba}{\mathbb{U}}

\newcommand{\mcurv}{\mathscr{MC}}

\newcommand{\quat}{\mathbb{H}}
\newcommand{\cayley}{\mathbb{O}}

\newcommand{\ideal}{\mathbb{I}}

\newcommand{\fie}{\mathbb{k}}
\newcommand{\hur}{\mathbb{h}}

\newcommand{\so}{\mathfrak{so}}

\newcommand{\chr}{\operatorname{char}}

\newcommand{\cone}{\mathbb{C}}
\newcommand{\squares}{\mathbb{Q}}

\newcommand{\alg}{\mathbb{A}}

\newcommand{\balg}{\mathbb{B}}
\newcommand{\calg}{\mathbb{U}}

\newcommand{\h}{\mathfrak{h}}

\newcommand{\om}{\omega}

\newcommand{\mlt}{\circ}

\renewcommand{\part}{\vdash}

\newcommand{\Id}{\operatorname{Id}}

\newcommand{\dum}{\,\cdot\,}

\newcommand{\cycle}{\mathsf{Cycle}}

\newcommand{\Q}{\mathcal{Q}}

\newcommand{\id}{\operatorname{Id}}

\renewcommand{\j}{i}

\newcommand{\la}{\lambda}
\newcommand{\ep}{\epsilon}

\newcommand{\fiet}{\dot{\fie}}
\newcommand{\reat}{\dot{\rea}}

\newcommand{\ext}{\mathsf{\Lambda}}
\newcommand{\cinf}{C^{\infty}}

\newcommand{\eno}{\operatorname{End}}

\newcommand{\si}{\sigma}

\newcommand{\sign}{\operatorname{sgn}}

\newcommand{\integer}{\mathbb{Z}}

\newcommand{\re}{\operatorname{Re}}
\newcommand{\im}{\operatorname{im}}

\newcommand{\lb}{\langle}
\newcommand{\ra}{\rangle}

\newcommand{\ste}{\mathbb{V}}

\newcommand{\spn}{\operatorname{Span}}

\newcommand{\su}{\mathfrak{su}}
\newcommand{\sll}{\mathfrak{sl}}

\newcommand{\proj}{\mathbb{P}}

\newcommand{\Aut}{\operatorname{Aut}}

\newcommand{\g}{\mathfrak{g}}

\newcommand{\ad}{\operatorname{ad}}

\newcommand{\tensor}{\otimes}

\newcommand{\rea}{\mathbb R}
\newcommand{\com}{\mathbb C}

\newcommand{\tr}{\operatorname{tr}}

\newcommand{\kwedge}{\owedge}

\renewcommand{\P}{\mathcal{P}}

\let\oldtocsection=\tocsection
\let\oldtocsubsection=\tocsubsection
\renewcommand{\tocsection}[2]{\hspace{0em}\oldtocsection{#1}{#2}}
\renewcommand{\tocsubsection}[2]{\hspace{1em}\oldtocsubsection{#1}{#2}}

\begin{document}

\title[Sectional nonassociativity of metrized algebras]{Sectional nonassociativity of metrized algebras}

\author{Daniel J.~F. Fox}
\address{Departamento de Matemática Aplicada\\ Escuela Técnica Superior de Arquitectura\\ Universidad Politécnica de Madrid\\Av. Juan de Herrera 4 \\ 28040 Madrid España}
\email{daniel.fox@upm.es} 



\begin{abstract}
The sectional nonassociativity of a metrized (not necessarily associative or unital) algebra is defined analogously to the sectional curvature of a pseudo-Riemannian metric, with the associator in place of the Levi-Civita covariant derivative. For commutative real algebras nonnegative sectional nonassociativity is usually called the Norton inequality, while a sharp upper bound on the sectional nonassociativity of the Jordan algebra of Hermitian matrices over a real Hurwitz algebra is closely related to the Böttcher-Wenzel-Chern-do Carmo-Kobayashi inequality. These and other basic examples are explained, and there are described some consequences of bounds on sectional nonassociativity for commutative algebras. A technical point of interest is that the results work over the octonions as well as the associative Hurwitz algebras.
\end{abstract}

\maketitle

\setcounter{tocdepth}{1}  

\tableofcontents

\section{Introduction}

For an algebra $(\alg, \mlt)$ equipped with a metric, meaning a nondegenerate symmetric bilinear form, that is invariant, meaning $h(x\mlt y, z) = h(x, y \mlt z)$ for all $x, y, z\in \alg$, the \emph{sectional nonassociativity} $\sect(x, y)$ of the $h$-nondegenerate subspace $\spn\{x, y\} \subset \alg$ is defined in \cite{Fox-simplicial} to be
\begin{align}\label{sectnadefined0}
\begin{split}
\sect(x, y) & = \frac{h(x\mlt x , y\mlt y) - h(x\mlt y, y \mlt x)}{h(x, x)h(y, y) -h(x, y)^{2} }.
\end{split}
\end{align}
(Because $\mlt$ need not be commutative or anticommutative, the order of $x$ and $y$ in \eqref{sectnadefined0} matters.)
The sectional nonassociativity is analogous to the sectional curvature of the Levi-Civita connection of a pseudo-Riemannian metric, with the multiplication and its associator in place of the covariant derivative and its curvature. Section \ref{analogysection} describes in detail this analogy between the multiplication of an algebra an torsion-free connection on a tangent bundle. The notion itself, and its basic properties, are given in Section \ref{sectnasection}. 

Quantitative bounds on the sectional nonassociativity of a metrized real algebra encompass apparently unrelated inequalities such as the Böttcher-Wenzel-Chern-do Carmo-Kobayashi inequality bounding the norm of a matrix commutator and the Norton inequality asserting the nonnegativity of the quadratic expression $h( x\mlt x, y \mlt y) - h(x\mlt y, x \mlt y)$ for certain metrized commutative algebras arising in the theory of vertex operator algebras whose most famous representative is the $196844$-dimensional Griess algebra on which the monster finite simple group acts by automorphisms. 

It is apparent from \eqref{sectnadefined0} that, for a commutative algebra, the Norton inequality may be rephrased as \emph{nonnegative sectional nonassociativity}. On the other hand, if $(\alg, [\dum, \dum], h)$ is a metrized anticommutative algebra (for example, a Lie or Malcev algebra with $h$ a multiple of its Killing form), \eqref{sectnadefined0} takes the form 
\begin{align}\label{acsect}
\sect(x, y) & = \frac{h([x, y], [x, y])}{h(x, x)h(y, y) -h(x, y)^{2} }.
\end{align}
If $h$ is positive definite, the right-hand side of \eqref{acsect} is bounded from above by $\frac{h([x, y], [x, y])}{h(x, x)h(y, y)}$ and, by the Cauchy-Schwarz inequality, this quantity is bounded from above. The Böttcher-Wenzel inequality asserts that for the algebra $\mat(n, \com)$ of $n\times n$ complex matrices metrized by the Frobenius norm this quantity is bounded from above by $2$ (which is better than the constant obtained by applying Cauchy-Schwarz), and characterizes the case of equality. The Böttcher-Wenzel inequality can be interpreted as an upper bound on the sectional nonassociativity of the Lie algebra of $n\times n$ matrices $\mat(n, \com)$ with the usual matrix commutator. 

For an anticommutative algebra, \eqref{acsect} is obviously nonnegative, with equality if and only if $[x, y] = 0$, so that characterization of the equality case in the lower bound amounts to characterization of commuting pairs of elements. 
On the other hand, for a commutative algebra, \eqref{sectnadefined0} can take both signs. These observations suggest looking for similar bounds for other metrized algebras, for example commutative algebras or other sorts of Lie algebras.

The extreme case of equal upper and lower bounds is that of constant sectional nonassociativity. This is closely related to the notion of projective associativity of commutative algebras studied in \cite{Fox-simplicial}. Although it is not solved here, the classification of metrized algebras with constant sectional nonassociativity appears an accessible problem. However, if anisotropy of the metric is not assumed, examples abound and it is still not clear what form the classification will take even in the simplest case of vanishing sectional nonassociativity. Essentially by definition, associative and alternative algebras have vanishing sectional nonassociativity. In Section \ref{examplesection}, it is shown that both a quadratic $2$-step nilpotent Lie algebra and a metrized antiflexible algebra have vanishing sectional nonassociativity, and nontrivial examples of such algebras are given. In particular, Example \ref{antiflexibleexample} exhibits a Euclidean metrized antiflexible algebra with vanishing sectional nonassociativity whose underlying Lie algebra is $\so(3)$ (which has constant sectional nonassociativity $1$ when metrized by the negative of its Killing form) and whose underlying commutative algebra has constant negative sectional nonassociativity. (It gives a similar example with underlying Lie algebra $\sll(2, \rea)$.)

Section \ref{examplesection} describes other examples of algebras having constant sectional nonassociativity. Example \ref{crossproductexample} explains that cross product algebras could be \emph{defined} as antisymmetric metrized algebras having constant negative sectional nonassociativity. 
Example \ref{3dexample} describes a one-parameter family of pairwise nonisomorphic $3$-dimensional metrized commutative algebras having constant sectional nonassociativities in $(-\infty, 1/4]$. 

Section \ref{boundssection} gives examples of algebras for which there can be established bounds on the sectional nonassociativities. Example \ref{voaexample} recalls a result of Miyamoto shows that the Griess algebras of OZ vertex operator algebras satisfy the Norton inequality. In this context it would be interesting to establish upper bounds. Example \ref{compositionsectexample} shows that symmetric composition algebras such as para-Hurwitz and Okubo algebras have sectional nonassociativity bounded between $-1$ and $1$. 

The remainder of Section \ref{boundssection} is devoted to the proof of Theorem \ref{hermsecttheorem}, which gives sharp upper and lower bounds on the sectional nonassociativities of a simple Euclidean Jordan algebra of rank at least $3$ together with characterization of the cases of equality in the bounds. These algebras are simply the Jordan algebras $\herm(n, \hur)$ of Hermitian matrices over the real Hurwitz algebras $\hur \in \{\rea, \com, \quat, \cayley\}$, where $n \geq 3$, and, in the case of the octonions, $\cayley$, attention is restricted to $n =3$. Theorem \ref{hermsecttheorem} states that the sectional nonassociativity, $\sect(x, y)$, defined in Definition \ref{sectdefinition}, of the subspace spanned by linearly independent $x, y \in \herm(n, \hur)$ satisfies
\begin{align}
0 \leq \sect(x, y) = \sect_{h, \star}(x, y) \leq \tfrac{n}{2},
\end{align}
and characterizes the subspaces for equality is obtained in the upper and lower bound. The lower bound is the Norton inequality.
The sharp upper bound is deduced as a consequence of the Chern-do Carmo-Kobayashi inequality for the commutators of Hermitian matrices. This requires some argument to extend this inequality to the octonions that is formulated as Lemma \ref{cdklemma}; this extension may be of independent interest.

These results are of interest from the point of view of applications of the Böttcher-Wenzel-Chern-do Carmo-Kobayashi inequality. This inequality is recalled in section \ref{bwsection} and it is shown that the Chern-do Carmo-Kobayashi inequality holds for $\herm(n, \quat)$ and partially for $\herm(3, \cayley)$. 
These refinements, Theorem \ref{bwtheorem} and the variant Lemma \ref{cdklemma}, valid over a real Hurwitz algebra, are used in the proof of Theorem \ref{hermsecttheorem} to estimate the sectional nonassociativity of the deunitalization of a simple Euclidean Jordan algebra. In a sense, this amounts to reinterpreting the  Böttcher-Wenzel-Chern-do Carmo-Kobayashi inequalities as upper bounds on sectional nonassociativity. 

The notion of sectional nonassociativity facilitates quantifications of nonassociativity both more refined and less exact than nonnegativity or nonpositivity, and these may in turn suggest refinements of the Böttcher-Wenzel inequality. More generally, the sectional nonassociativity is manifestly invariant under isometric algebra automorphisms and its relation with other similar quantities deserves to be explored more. See Remark \ref{vinbergremark} for further comments in this direction. 

Not many examples of commutative algebras satisfying the Norton inequality are known and the description of such algebras is interesting and their classification, possibly with additional hypotheses, may be a possibly tractable problem. The main known examples are the Griess algebras of certain vertex operator algebras \cite{Miyamoto-griessalgebras} and their subalgebras, as is explained next in Example \ref{voaexample}, and, by Theorem \ref{hermsecttheorem}, the simple Euclidean Jordan algebras, as is shown in Theorem \ref{hermsecttheorem}. (It is known that most of the simple Euclidean Jordan algebras can also be obtained as Griess algebras of certain vertex operator algebras \cite{Ashihara-Miyamoto, Lam-jordan, Zhao-jordan}.) Further examples are the Majorana algebras introduced by A.~A. Ivanov \cite{Ivanov, Ivanov-Pasechnik-Seres-Shpectorov} which are defined by a list of axioms including the Norton inequality. See \cite{Whybrow} for a family of examples depending on a parameter and further references.

More generally, the notion of sectional nonassociativity should help to organize the study of metrized algebras in the same manner that restriction on curvature quantities help to organize the study of Riemannian manifolds. Typically classes of algebras are defined by identities. From the point of view advocated here this is something like studying symmetric spaces (whose curvature tensors satisfy algebraic equalities). 

Section \ref{consequencessection} presents some general consequences for commutative algebras of bounds on sectional nonassociativity. Combined with the notion of sectional nonassociativity, results about the eigenvectors and eigenvalues of the structure tensor can be used to relate the existence and description of idempotents and square-zero elements to conditions involving sectional nonassociativity. Since an automorphism of the algebra must permute its idempotents and square-zero elements, such results yield results about the size of the automorphism group. For example, Theorem \ref{finiteautotheorem} shows that a Euclidean metrized commutative algebra has finite automorphism group if its sectional nonassociativity is negative or if its sectional nonassociativity is nonpositive and the algebra contains no trivial subalgebra. In the other direction, Lemma \ref{rchassoclemma} shows that nonnegative sectional nonassociativity precludes existence of $2$-nilpotents. The flavor of such results conforms broadly with the geometric analogy. For example, the automorphism group of a negatively curved Riemann surface is finite, while the (positively curved) sphere has a large continuous group of automorphisms.  

The paper is structured as follows. Section \ref{preliminarysection} presents basic definitions related to metrized algebras. Section \ref{analogysection} shows that the analogy between multiplications and covariant derivatives has substance and presents results that help motivate the definition of sectional nonassociativity given in Section \ref{sectnasection}. Section \ref{sectnasection} also describes the most basic properties of the sectional nonassociativity. Section \ref{examplesection} presents examples having constant sectional nonassociativity; these include Hurwitz algebras, cross-product algebras, antiflexible algebras, and some novel three-dimensional algebras. Section \ref{boundssection} describes examples satisfying bounds on sectional nonassociativity generalizing the Norton inequality; these include symmetric composition algebras, Griess algebras of OZ vertez operator algebras, and the simple Euclidean Jordan algebras. The proof of the upper bound on the sectional nonassociativity of simple Euclidean Jordan algebras given in Theorem \ref{hermsecttheorem} requires  the Böttcher-Wenzel-Chen-do Carmo-Kobayashi inequality for matrices over a Hurwitz algebra. Section \ref{bwsection} describes this and addresses what happens over the octonions. Finally, Section \ref{consequencessection} presents some general consequences for commutative algebras of bounds on sectional nonassociativity. 

\section{Preliminaries}\label{preliminarysection}
Let $\fie$ be a field. Everywhere in the paper there is assumed $\chr \fie \notdivides 6$, although sometimes assumptions on $\chr \fie$ are stated explicitly for clarity.

An algebra means a finite-dimensional vector space $\alg$ over $\fie$ equipped with a bilinear multiplication $\mlt:\alg \times \alg \to \alg$. It need not be associative, commutative, or unital. The left and right multiplication endomorphisms $L_{\mlt}, R_{\mlt}:\alg \to \eno(\alg)$ are defined by $L_{\mlt}(x)y = x\mlt y = R_{\mlt}(y)x$. 

An algebra is \emph{metrized} if it is equipped with a \emph{metric}, meaning a nondegenerate symmetric bilinear form, $h$, that is \emph{invariant} (also called \emph{associative}), meaning that $h(x\mlt y, z) = h(x, y\mlt z)$ for all $x, y, z\in \alg$. 

A \emph{Euclidean} metrized algebra is a real algebra metrized by a positive definite metric. (Essentially all results in the paper stated for Euclidean metrized algebras remain true with a real-closed base field in place of $\rea$, but it would be distracting to state and prove them in this generality.) That an algebra be Euclidean metrized means that the adjoint of the left multiplication endomorphism $L_{\mlt}(x)$ equals the right multiplication endomorphism $R_{\mlt}(x)$. Consequently, for a Euclidean metrized commutative algebra the multiplication endomorphisms $L_{\mlt}(x) = R_{\mlt}(x)$ are self-adjoint with respect to a positive definite symmetric bilinear form. 

\begin{example}
A semisimple real Lie algebra $(\g, \mlt)$ is metrized by the negative $h = -B_{\g}$ of its Killing form $B_{\g} = \tr L_{\mlt}(x)L_{\mlt}(y)$. It is Euclidean in the above sense if it is moreover compact. More generally, an anticommutative algebra $(\alg, \mlt)$ is \emph{Malcev} if
\begin{align}\label{malcev}
&(x\mlt y)\mlt (x\mlt z) = ((x\mlt y)\mlt z)\mlt x + ((y\mlt z)\mlt x)\mlt x + ((z\mlt x)\mlt x)\mlt y,&& \text{for all}\,\, x, y, z \in \alg.
\end{align}
for all $x, y, z \in \alg$. Any Lie algebra is Malcev. By \cite[Theorem $7.16$]{Sagle-malcev} the Killing form $\tau_{\mlt}(x, y) = \tr L_{\mlt}(x)L_{\mlt}(y)$ of a Malcev algebra is invariant. The identity \eqref{malcev} can be written $L_{\mlt}(x\mlt y)L_{\mlt}(x) + L_{\mlt}(x)L_{\mlt}(x\mlt y) = [L_{\mlt}(x)^{2}, L_{\mlt}(y)]$ and tracing this shows $\tau_{\mlt}(x\mlt y, x) = \tau_{\mlt}(x, y\mlt x)$. Polarizing this in $x$ yields the claim.
\end{example}

Lemma \ref{2dfrobcommutativelemma} generalizes the statement that a Lie algebra admitting an invariant nondegenerate bilinear form has dimension at least $3$.
\begin{lemma}\label{2dfrobcommutativelemma}
A two-dimensional metrized algebra $(\alg, \mlt, h)$ is commutative.
\end{lemma}
\begin{proof}
If $u$ and $v$ span $\alg$, then $h(u \mlt v, u) = h(u, u \mlt v) = h(u\mlt u, v) = h(v, u \mlt u) = h(v \mlt u, u)$ and, similarly, $h(u\mlt v, v) = h(v \mlt u, v)$. By the nondegeneracy of $h$ this implies $u\mlt v = v \mlt u$, which suffices to show the claim.
\end{proof}

The \emph{associator} of an algebra is the trilinear map $\tensor^{3}\alg^{\ast} \to \alg$ defined by 
\begin{align}
[x, y, z] = \ass_{\mlt}(x, y)z = (L_{\mlt}(x\mlt y) - L_{\mlt}(x)L_{\mlt}(y))z = (x\mlt y)\mlt z - x\mlt (y \mlt z).
\end{align} 
The map $\ass_{\mlt}(x, y) \in \eno(\alg)$ is the \emph{(left) associator endomorphism}. 

\begin{example}
A commutative algebra is \emph{Jordan} if $[x\mlt x, y, x] = 0$ for all $x, y, \in \alg$. For a Jordan algebra, the bilinear form $\tr L_{\mlt}(x\mlt y)$ is invariant, so it is metrized if this bilinear form is nondegenerate. A real Jordan algebra is Euclidean if it is formally real, meaning that a sum of squares equals zero if and only if each of the elements squares is zero. A real Jordan algebra is Euclidean as a Jordan algebra if and only if $\tr L_{\mlt}(x\mlt y)$ is positive definite \cite{Jacobson-jordan}.
\end{example}

\begin{example}
A \emph{composition algebra} is an algebra $(\alg, \mlt, q)$ with a nondegenerate quadratic form $q$ that is multiplicative meaning $q(x\mlt y) = q(x)q(y)$ for all $x, y \in \alg$. A \emph{Hurwitz} algebra is a unital composition algebra $(\alg, \mlt, e, q)$. By the Hurwitz theorem, a real Hurwitz algebra is one of the real field $\rea$, the complex field $\com$, the quaternions $\quat$, or the octonions $\cayley$.  For background on composition and Hurwitz algebras see \cite{Elduque-composition, Knus-Merkurjev-Rost-Tignol, Shapiro-compositions}.

A Hurwitz algebra carries the nondegenerate symmetric bilinear form $h(x, y) = q(x + y) - q(x) - q(y)$ obtained from $q$ by linearization. By the nondegeneracy and multiplicativity of $q$, $h(e, e) = 2q(e) = 2$. The \emph{standard involution} of $(\alg, \mlt, e, q)$ is the algebra antiautomorphism defined by $\bar{x} = h(x, e)e - x$. Its fixed point set is the subalgebra generated by $e$ and is isomorphic to $\rea$. Taking $z = e$ yields $h(x, y)e = x\bar{y} + y \bar{x} = h(\bar{x}, \bar{y})e$. A Hurwitz algebra equipped with its standard involution is a metrized algebra with involution in the sense of \cite{Allison}, meaning that $h$ is \emph{involutively invariant} in the sense that $h(x, z\mlt \bar{y}) = h(x\mlt y, z) = h(y, \bar{x}\mlt z)$ for all $x, y, z \in \alg$. Note that, except when $\hur = \fie$, $h$ is \emph{not} invariant, so $(\hur, \mlt, h)$ is \emph{not} metrized as an algebra.
\end{example}

\section{The analogy between multiplications and connections}\label{analogysection}
The following analogy between algebras and connections illuminates the point of view taken here. The multiplication $\mlt$ is analogous to a torsion-free connection on a tangent bundle, and the associator corresponds to (the negative of) the second covariant derivative. In fact, a torsion-free connection on a tangent bundle can be regarded as a Lie-admissible product on the space of vector fields, and in that setting the sectional nonassociativity defined in Section \ref{sectnasection} \emph{is} simply the sectional curvature. However that setting differs from that considered here in several ways. First, the space of vector fields is infinite-dimensional. Second, the covariant derivative is a $\cinf$-module map in its first argument. A general context encompassing both finite-dimensional algebras and connections on tangent bundles is algebroids, but this is not discussed here. 

The notions suggested by the multiplication-connection analogy prove surprisingly robust and provide principles for organizing the zoo of general algebras different from those that arise from classical points of view. In particular, they offer possibilities for quantifying nonassociativity that turn out to be quite useful.

This idea has been discussed in similar terms in \cite{Goze-Remm-lieadmissible} in the context of Lie-admissible algebras and has roots in foundational articles on left-symmetric algebras \cite{Shima, Vinberg}. In \cite{Bremner-Dotsenko} there are considered algebras satisfying identities like those treated in this section from the more sophisticated point of view of Koszul duality of operads.  The unrelated duality given by the adjoint multiplication considered here is motivated by considerations from affine differential geometry \cite{Fox-crm}. These ideas are expounded in the more limited context of metrized commutative algebras in \cite{Fox-simplicial, Fox-cubicpoly}.

The analogy is implemented in practice in the following way. The space of connections on a bundle is an affine space. A reference connection is fixed and then any other connection is identified with a tensor. Any expression involving this connection involves some derivatives and some purely algebraic part expressible in terms of this tensor. The purely algebraic expression make sense also for the structure tensor of an algebra. For example, the covariant derivative of a metric $g$ with respect to a torsion-free affine connection $\nabla$ is $(\nabla_{X}g)(Y, Z) = X(g(Y, Z)) - g(\nabla_{X}Y, Z) - g(Y, \nabla_{X}Z)$. The corresponding expression for a metric $g$ on an algebra $(\alg, \mlt)$ is $-g(x\mlt y, z) - g(y, x \mlt z)$. The vanishing of this expression is analogous to the metric tensor being parallel. This kind of invariance should probably be properly formulated as a cohomological condition, but it is not clear how to give sense to this last claim when working in the generality considered here.

In the analogy, the underlying bracket is regarded as background information of a topological rather than geometric nature, analogous to the Lie algebra of vector fields on a smooth manifold.

Applying the connection-algebra analogy to the curvature of an affine connection yields the following definitions.

Decomposing the product $\mlt$ by symmetries leads to two underlying algebras, from which $(\alg, \mlt)$ can be reconstructed (except when $\chr \fie = 2$). The \emph{underlying bracket} (or \emph{commutator}) of the algebra $(\alg, \mlt)$ is the commutator defined by $[x, y] = x\mlt y - y \mlt x$ and $(\alg, [\dum, \dum])$ is the \emph{underlying} antisymmetric algebra.
The \emph{underlying commutative algebra} of an algebra $(\alg, \mlt)$ is $\alg$ equipped with the product $x \sprod y = x\mlt y + y \mlt x$. 

The multiplication $\amlt$ \emph{adjoint} to $\mlt$ is that defined by $x \amlt y = - y \mlt x$. It is the negative of the opposite multiplication. By definition, that $L_{\amlt} = - R_{\mlt}$ and $R_{\amlt} = -L_{\mlt}$. The sign in the definition of the adjoint multiplication $\amlt$ is chosen so that $\mlt$ and the adjoint product $\amlt$ have the same underlying bracket.

The \emph{curvature} $\acl:\alg \tensor \alg \to \eno(\alg)$ and \emph{adjoint curvature} $\acr:\alg \tensor \alg \to \eno(\alg)$ of $(\alg, \mlt)$ are defined by 
\begin{align}\label{acdefined}
\begin{split}
\acl(x, y)z & = -[x, y, z] + [y, x, z]= \left([L_{\mlt}(x), L_{\mlt}(y)] - L_{\mlt}([x, y])\right)z  = (\ass_{\mlt}(y, x) - \ass_{\mlt}(x, y))z,\\ 
\acr(x, y)z &= -[z, x, y] + [z, y, x] = \left([R_{\mlt}(x), R_{\mlt}(y)] + R_{\mlt}([x, y])\right)z= (\ass_{\amlt}(y, x) - \ass_{\amlt}(x, y))z.
\end{split}
\end{align}
The adjoint curvature of $\mlt$ is the curvature of the adjoint multiplication $\amlt$ and vice-versa. For any tensorial object constructed from a multiplication $\mlt$ there is a corresponding tensorial object constructed in the same way from the adjoint multiplication $\amlt$, and it is denoted by putting a $\bar{\,}$ over the notation indicating the object associated with $\mlt$.  This convention is consistent with the notation $\acl$ and $\acr$.

The algebra has \emph{self-adjoint curvature} if $\acl = \acr$ and \emph{anti-self-adjoint curvature} if $\acl  = -\acr$.

\begin{example}
An algebra $(\alg, \mlt)$ is \emph{flexible} if $[x, y, x] = 0$ for all $x, y \in \alg$. A commutative or anticommutative algebra is flexible. The symmetric composition algebras discussed in Example \ref{compositionsectexample} are also flexible.
It is evident from \eqref{acdefined} that a flexible algebra has self-adjoint curvature.
\end{example}

\begin{example}\label{antiflexibleasaexample}
An algebra $(\alg, \mlt)$ is \emph{antiflexible} if $[x, y, z] = [z, y, x]$ for all $x, y, z \in \alg$. It is evident from \eqref{acdefined} that an antiflexible algebra has anti-self-adjoint curvature.
\end{example}

\begin{example}
An algebra is \emph{left-symmetric} (also called \emph{pre-Lie} or \emph{Vinberg}) if $\ass_{\mlt}(x, y) = \ass_{\mlt}(y, x)$ for all $x, y \in \alg$. It is evident from \eqref{acdefined} that an algebra is left-symmetric if and only if its curvature $\acl$ vanishes. The vanishing of the adjoint curvature is equivalent to the left symmetry of the adjoint multiplication, or that the given multiplication is \emph{right-symmetric}. 

Any left-symmetric algebra that is not right-symmetric gives an example of an algebra $(\alg, \mlt)$ not isomorphic to its adjoint $(\alg, \amlt)$. More generally, an algebra that does not have self-adjoint curvature is not isomorphic to its adjoint algebra.
\end{example}

\begin{example}\label{2stepflatexample}
An awkward aspect of the conventions here is that the underlying bracket of an anticommutative algebra is twice the original bracket, and this needs to be remembered when interpreting \eqref{acdefined}. 

As a Lie algebra $(\g, [\dum, \dum])$ is flexible, it has self-adjoint curvature, and from \eqref{acdefined} it follows that $\acr(x, y) = \acl(x, y) = -\ad_{\g}([x, y]) = -[\ad_{\g}(x), \ad_{\g}(y)]$. Consequently any $2$-step nilpotent Lie algebra provides an example of a nonassociative algebra for which both $\acl$ and $\acr$ vanish identically. 
\end{example}

Define $A_{\mlt}(x) = L_{[\dum, \dum]}(x) = L_{\mlt}(x) - R_{\mlt}(x)$, so that $A_{\mlt}(x)y = [x, y]$, and define $S_{\mlt}(x) = L_{\sprod}(x) = L_{\mlt}(x) + R_{\mlt}(x)$, so that $S_{\mlt}(x)y = x\sprod y$. Straightforward computations using \eqref{acdefined} show that, in any algebra,
\begin{align}\label{ascommutator}
[A_{\mlt}(x), S_{\mlt}(y)]- S_{\mlt}(A_{\mlt}(x)y) & = \acl(x, y) - \acr(x, y) - [R_{\mlt}(x), L_{\mlt}(y)] - [R_{\mlt}(y), L_{\mlt}(x)].
\end{align}
Decomposing \eqref{ascommutator} into its parts symmetric and antisymmetric in $x$ and $y$ yields the relations
\begin{align}\label{preflexible}
\begin{split}
[A_{\mlt}(x), S_{\mlt}(y)] + [A_{\mlt}(y), S_{\mlt}(x)] = -2[R_{\mlt}(x), L_{\mlt}(y)] - 2[R_{\mlt}(y), L_{\mlt}(x)],
\end{split}
\end{align}
and
\begin{align}\label{structurable}
\begin{split}
\acl(x, y)x - \acr(x, y)z & = -[x, y, z] + [z, x, y] - [z, y, x] + [y, x, z] \\
& = \left([L_{\mlt}(x), R_{\mlt}(z)] + [L_{\mlt}(z), R_{\mlt}(x)]\right)y - \left([L_{\mlt}(y), R_{\mlt}(z)] + [R_{\mlt}(y), L_{\mlt}(z)]\right)x,\\\
& = \left(\tfrac{1}{2}[A_{\mlt}(x), S_{\mlt}(y)] - \tfrac{1}{2}[A_{\mlt}(y), S_{\mlt}(x)] - S_{\mlt}([x, y])\right)z,
\end{split}
\end{align}
for all $x, y, z \in \alg$ (the validity of the last equality needs $\chr \fie \neq 2$). 

\begin{lemma}\label{flexiblelemma}
If $\chr \fie \neq 2$, the following conditions on a $\fie$-algebra $(\alg, \mlt)$ are equivalent.
\begin{enumerate}
\item\label{ppo3} $[A_{\mlt}(x), S_{\mlt}(x)] = 0$ for all $x \in \alg$.
\item\label{rlanti} $[R_{\mlt}(x), L_{\mlt}(y)] = -[R_{\mlt}(y), L_{\mlt}(x)]$ for all $x, y \in \alg$.
\item\label{ppo2} $(\alg, \mlt)$ is flexible.
\end{enumerate}
Over a field of any characteristic, a flexible algebra satisfies \eqref{ppo3} and \eqref{rlanti} and has self-adjoint curvature.
\end{lemma}
\begin{proof}
By definition, $(\alg, \mlt)$ is flexible if and only if $[R_{\mlt}(x), L_{\mlt}(x)] =0$ for all $x \in \alg$. By \eqref{preflexible} this implies \eqref{ppo3} and is implied by \eqref{ppo3} if $\chr \fie \neq 2$. Polarization shows any flexible algebra satisfies \eqref{rlanti}, and if $\chr \fie \neq 2$, then the validity of \eqref{rlanti} implies $(\alg, \mlt)$ is flexible (taking $y  = x$ in \eqref{rlanti}). 
By \eqref{rlanti} and \eqref{structurable}, a flexible algebra has self-adjoint curvature.
\end{proof}

\begin{lemma}\label{prepoissonlemma}
If $\chr \fie \neq 2$, a $\fie$-algebra $(\alg, \mlt)$ is flexible if and only if it has self-adjoint curvature and $A_{\mlt}(x)$ is a derivation of the underlying commutative algebra $(\alg, \sprod)$ for all $x \in \alg$.
\end{lemma}
\begin{proof}
This is immediate from Lemma \ref{flexiblelemma} and \eqref{ascommutator}.
\end{proof}

\begin{example}
By \cite[Proposition $1$]{Allison}, the right-hand side of \eqref{structurable} vanishes for any three elements in the subspace of a structurable algebra fixed by its involution. 
\end{example}

The analogue of the algebraic Bianchi identity for a curvature tensor is the vanishing of the complete antisymmetrization of $\acl(x, y)z$. This need not occur in general for, as is the case for an affine connection, there is an obstruction related to compatibility with an underlying Lie bracket.

An algebra $(\alg, \mlt)$ is \emph{Lie-admissible} if the underlying bracket $[\dum, \dum]$ satisfies the Jacobi identity, so that $(\alg, [\dum, \dum])$ is a Lie algebra. That $(\alg, \mlt)$ be Lie-admissible is equivalent to the vanishing of the completely antisymmetric tensor 
\begin{align}\label{lieadmissible}
\begin{split}
\at(x_{1}, x_{2}, x_{3}) &= \left(A_{\mlt}([x_{1}, x_{2}]) - [A_{\mlt}(x_{1}), A_{\mlt}(x_{2})]\right)x_{3} = \cycle_{1, 2, 3}[[x_{1}, x_{2}], x_{3}] \\
&= \sum_{\si \in S_{3}}\sign(\si)[x_{\si(1)}, x_{\si(2)}, x_{\si(3)}] = - \cycle_{1, 2, 3}\acl(x_{1}, x_{2}) x_{3}= - \cycle_{1, 2, 3}\acr(x_{1}, x_{2})x_{3}.
\end{split}
\end{align}
(The notation $\cycle_{1, 2, 3}$ indicates the sum over the cyclic permutations of the indicated indices.) The validity of the identity $- \cycle_{1, 2, 3}\acl(x_{1}, x_{2}) x_{3}=\cycle \left([x_{1}, x_{2}, x_{3}] - [x_{2}, x_{1}, x_{3}]\right) = \cycle [[x_{1}, x_{2}], x_{3}]$ in any algebra is noted in \cite[Equation $2.19$]{Okubo-octonion}. That the vanishing of the first expression on the second line of \eqref{lieadmissible} is equivalent to Lie admissibility is stated in \cite[Example $6$]{Markl-Remm}.

Note that $\bar{\at} = \at$.
The final two equalities of \eqref{lieadmissible} show that the identity $\at = 0$ is analogous to the algebraic Bianchi identity for a torsion-free connection, so the tensor $\at$ should be regarded as the algebraic analogue of the torsion of an affine connection. 

A straightforward computation shows $\at$ is the self-adjoint part of 
\begin{align}\label{aclmu}
\begin{split}
[A_{\mlt}(x), L_{\mlt}(y)] - L_{\mlt}(A_{\mlt}(x)y) 
\end{split}
\end{align}
which vanishes for all $x, y \in \alg$ if and only if $A_{\mlt}(x)$ is a derivation of $(\alg, \mlt)$ for all $x \in \alg$.

\begin{example}\label{antiflexibleadmissibleexample}
It is immediate from \eqref{lieadmissible} that $\at = 0$ for an antiflexible algebra, so an antiflexible algebra is Lie-admissible.
\end{example}

The analogue of the covariant derivative of the curvature tensor, $(\nabla_{X}R)(Y, Z)$, is $[L_{\mlt}(x), \acl(y, z)] - \acl(L_{\mlt}(x)y, z) - \acl(y, L_{\mlt}(x)z)$. The identities \eqref{diffbianchi} are analogous to the differential Bianchi identity, and reinforce the idea that $\at$ should be viewed as something like the torsion of $\mlt$. 
\begin{lemma}
In any algebra $(\alg, \mlt)$, for all $x, y, z \in \alg$ there hold
\begin{align}\label{diffbianchi}
\begin{aligned}
\cycle_{x, y, z}&\left([L_{\mlt}(x), \acl(y, z)] - \acl(L_{\mlt}(x)y, z) - \acl(y, L_{\mlt}(x)z) \right) = L_{\mlt}(\at(x, y, z)),\\
\cycle_{x, y, z}&\left([R_{\mlt}(x), \acr(y, z)] - \acr(R_{\mlt}(x)y, z) - \acr(y, R_{\mlt}(x)z) \right) = R_{\mlt}(\at(x, y, z)).
\end{aligned}
\end{align}
\end{lemma}
\begin{proof}
A computation using \eqref{acdefined}, the Jacobi identity in $\eno(\alg)$, the identity $[x\mlt y, z] - [x, z\mlt y] = [x, y, z] - [z, y, x]$, and \eqref{lieadmissible} shows the first identity of \eqref{diffbianchi}.
Taking $\amlt$ in place of $\mlt$ and recalling $\bar{\at} =\at$ yields the second identity of \eqref{diffbianchi}.
\end{proof}

An algebra $(\alg, \mlt)$ is \emph{associator-cyclic} if it satisfies the identity
\begin{align}\label{commassoc}
0 = [x, y, z] + [z, x, y] + [y, z, x] = \left(A_{\mlt}(x\mlt y) - A_{\mlt}(x)L_{\mlt}(y) - A_{\mlt}(y)R_{\mlt}(x)\right)z,
\end{align}
for all $x ,y, z \in \alg$.
(There seems to be no standard terminology for algebras satisfying \eqref{commassoc}, and \emph{associator-cyclic} seems reasonably evocative terminology.)

\begin{example}
Commutative algebras and Lie algebras are associator-cyclic, but a general anticommutative algebra need not be associator-cyclic.
\end{example}

\begin{lemma}
If $\chr \fie \neq 3$, an associator-cyclic $\fie$-algebra has self-adjoint curvature if and only if it is flexible.
\end{lemma}
\begin{proof}
By Lemma \ref{flexiblelemma}, a flexible algebra has self-adjoint curvature, so it suffices to prove that an associator-cyclic algebra with self-adjoint curvature is flexible. Taking $x = z$ in \eqref{commassoc} and \eqref{structurable} yields $[x, y, x] = -[x, x, y] - [y, x, x] = -2[x, y, x]$, so that $3[x, y, x] = 0$ for all $x, y \in \alg$. If $\chr\fie \neq 3$, this implies $(\alg, \mlt)$ is flexible.
\end{proof}

Antisymmetrizing \eqref{commassoc} in $x$ and $y$ shows that an associator-cyclic algebra is Lie-admissible. Lemma \ref{lieadmissiblecycliclemma} shows that these conditions are equivalent for flexible $\fie$-algebras when $\chr \fie \neq 2$. That when $\chr \fie \neq 2$ an algebra is flexible and Lie-admissible if and only if there holds \eqref{cycadm2} of Lemma \ref{lieadmissiblecycliclemma} is due to Okubo and Myung in \cite[p. $80$]{Myung} and \cite[Theorem $4.1$]{Okubo-Myung-flexible-adjoint}.
\begin{lemma}[{\cite[p. $80$]{Myung}}, {\cite[Theorem $4.1$]{Okubo-Myung-flexible-adjoint}}, \cite{Myung-book}]\label{lieadmissiblecycliclemma}
For a $\fie$-algebra $(\alg, \mlt)$ the following are equivalent.
\begin{enumerate}
\item\label{cycadm1} $(\alg, \mlt)$ is flexible and associator-cyclic.
\item\label{cycadm2} For all $x \in \alg$, $A_{\mlt}(x) = L_{\mlt}(x) - R_{\mlt}(x)$ is a derivation of $(\alg, \mlt)$.
\item\label{cycadm3} $(\alg, \mlt)$ has self-adjoint curvature. 
\end{enumerate}
These imply
\begin{enumerate}
\setcounter{enumi}{3}
\item\label{cycadm4} $(\alg, \mlt)$ is flexible and Lie-admissible.
\end{enumerate}
If $\chr \fie \neq 2$, then \eqref{cycadm4} implies \eqref{cycadm1} and \eqref{cycadm2}.
In particular, if $\chr \fie \neq 2$, a flexible algebra is Lie-admissible if and only if it is associator-cyclic.
\end{lemma}
\begin{proof}
For any algebra there holds
\begin{align}\label{cyclicderiv}
\begin{split}%
\cycle_{x, y, z}[x, y, z] 
&= \left([A_{\mlt}(x), L_{\mlt}(y)] - L_{\mlt}(A_{\mlt}(x)y) \right)z + [L_{\mlt}(x), R_{\mlt}(z)]y.
\end{split}
\end{align}
It is immediate from \eqref{cyclicderiv} that if $(\alg, \mlt)$ is flexible and associator-cyclic, then it satisfies \eqref{cycadm2}. If $(\alg, \mlt)$ satisfies \eqref{cycadm2}, then $-[R_{\mlt}(x), L_{\mlt}(x)] = [A_{\mlt}(x), L_{\mlt}(x)] = L_{\mlt}(A_{\mlt}(x)x) = 0$ for all $x \in \alg$, so $(\alg, \mlt)$ is flexible, and with \eqref{cyclicderiv} it follows that $(\alg, \mlt)$ is associator-cyclic. This proves the equivalence of \eqref{cycadm1} and \eqref{cycadm2}. The equivalence of \eqref{cycadm2} and \eqref{cycadm3} is immediate from \eqref{aclmu} and its adjoint relation. 

By \eqref{lieadmissible}, $\at(x, y, z) = \cycle_{x, y, z}[x, y, z] - \cycle_{x, y, z}[y, x , z]$ (alternatively, by antisymmetrizing \eqref{commassoc} in $x$ and $y$), so an associator-cyclic algebra is Lie-admissible, and \eqref{cycadm1} implies \eqref{cycadm4}.
If $(\alg, \mlt)$ is Lie-admissible and flexible, then $0 = \cycle_{x, y, z}([x, y, z] - [y, x , z]) = 2\cycle_{x, y, z}([x, y, z])$, so, if $\chr \fie \neq 2$, $(\alg, \mlt)$ is associator-cyclic.
\end{proof}

\begin{remark}
The essential content of Lemma \ref{lieadmissiblecycliclemma} is due to Okubo and Myung \cite{Myung, Myung-book, Okubo-Myung-flexible-adjoint}. It is formulated here in a manner consistent with the perspective taken here. The equivalence of \eqref{cycadm1} and \eqref{cycadm2} and that these imply \eqref{cycadm3} is also proved, with different terminology, in \cite{Goze-Remm-poisson} (see also \cite[Propositions $1$ and $2$]{Remm-Morand}). (In \cite{Remm-Morand} an algebra is called \emph{admissible Poisson} if there holds \eqref{cycadm2} and \emph{weakly associative} if there vanishes the right-hand side of \eqref{cyclicderiv} (so it satisfies \eqref{cycadm3} of Lemma \ref{lieadmissiblecycliclemma}).)
\end{remark}

\begin{example}
By the Jacobi identity, the associator of a Lie algebra $(\g, [\dum, \dum])$ is $[x, y, z] = [y, [x, z]]$ and a Lie algebra is associator-cyclic. This shows an associator-cyclic, flexible algebra need not be commutative.
\end{example}

\begin{example}
By \cite{Schafer-cayleydickson} an algebra obtained from a flexible algebra by the Cayley-Dickson process is flexible. 
An example of a flexible algebra that is not Lie-admissible (so not associator-cyclic) is the imaginary octonions $\im \cayley$ equipped with the commutator bracket $[x, y] = xy - yx$. (Although $\im \cayley$ is Malcev-admissible \cite{Myung-book}.)
\end{example}

\section{Sectional nonassociativity}\label{sectnasection}
This section introduces quantitative notions of nonassociativity for metrized algebras. The sectional nonassociativity was introduced in \cite{Fox-simplicial} for metrized commutative algebras.

In the first part of this section the base field $\fie$ has characteristic not equal to $2$ or $3$. Let $\alg$ be an $n$-dimensional $\fie$-vector space with a metric $h$. Define
\begin{align}
\begin{aligned}
S^{2}\ext^{2}\alg^{\ast} &= \left\{\begin{aligned} a &\in \tensor^{4}\alg^{\ast}:\\
& a(x, y, z, w) = -a(y, x, z, w) = -a(x, y, w, z) = a(z, w, x,y) \,\,\text{for all}\,\, x, y, z, w \in \alg
\end{aligned}\right\},
\\
\mcurv(\alg) &= \{a \in S^{2}\ext^{2}\alg^{\ast}:a(x, y, z, w) + a(y, z, x, w) + a(z, x, y, w) = 0 \,\,\text{for all}\,\, x, y, z, w \in \alg\}.
\end{aligned}
\end{align}
The subspace $\mcurv(\alg)$ comprises tensors of metric curvature tensor type. 
An element $a \in S^{2}\ext^{2}\alg^{\ast}$ can be viewed as the bilinear form on $\ext^{2}\alg$ defined on decomposable elements by $a(x \wedge y, z \wedge w) = 4a(x, y, z, w)$, so determines an associated quadratic form on $\ext^{2}\alg$.
Endow $\tensor^{4}\alg^{\ast}$ and its subspaces with the metrics given by complete contraction with the metric $h$. If $\chr \fie \neq 3$, $S^{2}\ext^{2}\alg^{\ast}$ is isomorphic to the orthogonal direct sum $\mcurv(\alg) \oplus \ext^{4}\alg^{\ast}$, and the orthogonal projections onto the factors, $\P:S^{2}\ext^{2}\alg^{\ast} \to \mcurv(\alg)$ and $\Q:S^{2}\ext^{2}\alg^{\ast} \to \ext^{4}\alg^{\ast}$ are defined for $a \in S^{2}\ext^{2}\alg^{\ast}$ by
\begin{align}\label{mcurvproj}
\begin{aligned}
\P(a)(x, y, z, w) &= \tfrac{1}{3}\left(2a(x, y, z, w) - a(y, z, x, w) - a(z, x, y, w)\right),\\
\Q(a)(x, y, z, w) & = \tfrac{1}{3}\left(a(x, y, z, w) +a(y, z, x, w) + a(z, x, y, w)\right).
\end{aligned}
\end{align}
The tensors $a$ and $\P(a)$ do not necessarily determine the same quadratic form on $\ext^{2}\alg$, but the quadratic forms they determine agree on decomposable elements of $\ext^{2}\alg$.

\begin{lemma}\label{s2ext2lemma}
Suppose $\chr \fie \notdivides 6$. Let $h$ be a metric on the $\fie$-vector space $\alg$ and let $\P:S^{2}\ext^{2}\alg^{\ast} \to \mcurv(\alg)$ be the orthogonal projection with respect to the inner products given by complete contraction with $h$. The quadratic forms determined on $\ext^{2}\alg$ by $a, b \in S^{2}\ext^{2}\alg^{\ast}$ coincide on decomposable elements of $\ext^{2}\alg$ if and only if $\P(a) = \P(b)$.
\end{lemma}

\begin{proof}
For $\om \in \ext^{2}\alg$ and $a \in S^{2}\ext^{2}\alg^{\ast}$, 
\begin{align}
\begin{split}
\lb a, \om \tensor \om\ra & = \lb \P(a), \om \tensor \om\ra + \lb \Q(a), \om \tensor \om\ra = \lb \P(a), \om \tensor \om\ra + \tfrac{1}{6}\lb \Q(a), \om \wedge \om\ra,
\end{split}
\end{align}
in which $\lb \dum, \dum \ra$ denotes the pairing between dual vector spaces. Because $\om \wedge \om = 0$ if $\om$ is decomposable, this shows that the restrictions to decomposable $2$-forms of the quadratic forms determined by $a$ and $\P(a)$ coincide, and, moreover, that if $Q(a) = 0$, then these quadratic forms are equal. 

On the other hand, that the quadratic forms determined on $\ext^{2}\alg$ by $a, b \in S^{2}\ext^{2}\alg^{\ast}$ coincide on decomposable elements means that $c = a - b \in S^{2}\ext^{2}\alg^{\ast}$ satisfies $4c(x, y, x, y) = \lb c, (x\wedge y)\tensor (x\wedge y)\ra = 0$ for all $x, y \in \alg$. Linearizing in $x$ and $y$ and using the symmetries of $c$ yields
\begin{align}
0 & = c(x, y, z, w) + c(z, y, x, w) + c(x, w, z, y) + c(z, w, x, y) = 2c(x, y, z, w) - 2c(y, z, x, w)
\end{align}
for all $x, y, z, w \in \alg$. Hence $c(x, y, z, w) = c(y, z, x, w) = c(z, x, y, w)$, and in \eqref{mcurvproj} this yields $\P(a) - \P(b) = \P(c) = 0$. 
\end{proof}

For a metrized algebra $(\alg, \mlt, h)$, it follows straightforwardly from the invariance of $h$ that 
\begin{align}\label{metrizedass}
&h([x, y, z], w) = - h([y, z, w], x) = h([z, w, x], y) = -h([w, x, y], z), && \text{for all}\,\, x, y, z, w \in \alg.
\end{align}
Define $\acl^{\flat}(x, y, z, w) = h(\acl(x, y)z, w)$ and $\acr^{\flat}(x, y, z, w) = h(\acr(x, y)z, w)$.
From \eqref{metrizedass} and \eqref{acdefined} it follows that
\begin{align}\label{aclacrmu}
&\acl^{\flat}(x, y, z, w)=- \acr^{\flat}(x, y, w, z),
&& \text{for all}\,\, x, y, z, w \in \alg.
\end{align}
By \eqref{aclacrmu} and \eqref{metrizedass},
\begin{align}\label{aclacrijkl}
\begin{aligned}
&\acl^{\flat}(x, y,z, w) + \acr^{\flat}(x, y,z, w) =\acl^{\flat}(x, y,z, w) - \acl^{\flat}(x, y,w, z) \\
& =\acl^{\flat}(z, w,x, y) - \acl^{\flat}(z, w,y, x)= \acl^{\flat}(z, w,x, y) + \acr^{\flat}(z, w,x, y)
\end{aligned}
\end{align}
so that $\acl^{\flat} + \acr^{\flat}$ is an element of $S^{2}\ext^{2}\alg^{\ast}$ (as it need not satisfy the algebraic Bianchi identity, it need not be in $\mcurv(\alg)$). This follows also from the alternative expression
\begin{align}\label{aclacr}
\begin{split}
\acl^{\flat}(x, y,z, w) &+ \acr^{\flat}(x, y, z, w) \\
&= -h([x, y], [z, w]) + h(y \mlt z, w\mlt x) - h(x \mlt z, w \mlt y) + h(z\mlt y, x \mlt w) - h(z\mlt x, y \mlt w),
\end{split}
\end{align}
which is evidently antisymmetric in $x$ and $y$, and in $z$ and $w$, and symmetric in the pairs $xy$ and $zw$.

\begin{lemma}\label{presectlemma}
Suppose $\chr \fie \notdivides 6$. A metrized $\fie$-algebra $(\alg, \mlt, h)$ is Lie-admissible if and only if $\acl^{\flat} + \acr^{\flat} \in \mcurv(\alg)$. 
\end{lemma}
\begin{proof}
By \eqref{diffbianchi}, the tensor $\mcl$ defined by $\mcl^{\flat} = 3\P(\acl^{\flat}+ \acr^{\flat}) \in \mcurv(\alg)$ satisfies
\begin{align}
\begin{split}\label{mcldefined}
\mcl(x, y)z
& = -2(\acl(x, y)z + \acr(x, y)z) + \acl(y, z)x + \acr(y, z)x + \acl(z, x)y + \acr(z, x)y) \\
&= -3(\acl(x, y)z +\acr(x, y)z) - 2\at(x, y, z),
\end{split}
\end{align}
for all $x, y, z \in \alg$.
The claim follows because $(\alg, \mlt)$ is Lie-admissible if and only if $\at$ vanishes identically.
\end{proof} 

Let $(\alg, \mlt, h)$ be a metrized algebra of dimension at least $2$. If $x, y \in \alg$ span a two-dimensional subspace $\ste \subset \alg$ and $\bar{x} = ax + by$ and $\bar{y} = cx + dy$ span the same subspace $\ste$, then 
\begin{align}\label{xychange1}
&h(\bar{x}, \bar{x})h(\bar{y}, \bar{y}) - h(\bar{x}, \bar{y})^{2} = (ad - bc)^{2}\left(h(x, x)h(y, y) - h(x, y)^{2}\right),\\
\label{xychange2}
&h(\bar{x}\mlt \bar{x}, \bar{y}\mlt \bar{y}) - h(\bar{x}\mlt \bar{y}, \bar{x}\mlt \bar{y}) = (ad - bc)^{2}\left(h(x\mlt x , y\mlt y) - h(x\mlt y, y \mlt x)\right),
\end{align}
from which it follows that \eqref{xychange2} divided by \eqref{xychange1} is well-defined and depends only the subspace $\ste$, provided the denominator is non-zero, which is the case as long $\ste$ is $h$-nondegenerate. (A symmetric bilinear form $h$ on a two-dimensional $\fie$-vector space $\ste$ is nondegenerate if and only if $h(x, x)h(y, y) - h(x, y)^{2} \neq 0$ for any basis $\{x, y\}$ of $\ste$; this is a special case of the statement that a symmetric bilinear form is nondegenerate if and only if the determinant of its Gram matrix with respect to any basis is nonvanishing \cite[Proposition I.$1.2$]{Lam-quadraticforms}.) These observations show that Definition \ref{sectdefinition} is well made. By \eqref{metrizedass},
\begin{align}\label{hxxyy}
&h([x, x, y], y) = -h([y, x, x], y), && \text{for all}\,\, x, y \in \alg.
\end{align}
The second equality in \eqref{sectnadefined} is \eqref{hxxyy}.

\begin{definition}\label{sectdefinition}
Let $(\alg, \mlt, h)$ be a metrized $\fie$-algebra of dimension at least $2$. Define the \emph{sectional nonassociativity} $\sect(x, y)$ of the $h$-nondegenerate subspace $\spn\{x, y\} \subset \alg$ to be
\begin{align}\label{sectnadefined}
\begin{split}
\sect(x, y) & =  \sect_{\mlt, h}(x, y)  = -\tfrac{1}{2}\frac{h(\acl(x, y) x, y) + h(\acr(x, y) x, y)}{h(x\mlt x , y\mlt y) - h(x, y)^{2}}\\
&=  \frac{h([x, x, y], y)}{h(x, x)h(y, y)  - h(x, y)^{2}}=\frac{-h([y, x, x], y)}{h(x, x)h(y, y)  - h(x, y)^{2}}= \frac{h(x\mlt x , y\mlt y) - h(x\mlt y, y \mlt x)}{h(x, x)h(y, y)  - h(x, y)^{2}}.
\end{split}
\end{align}
\end{definition}
The proviso that $\spn\{x, y\}$ be $h$-nondegenerate is automatic when $h$ is anisotropic.

By definition, the sectional nonassociativity is invariant under isometric algebra isomorphisms of metrized algebras.

The dependence of $\sect =  \sect_{\mlt, h}$ on $\mlt$ and $h$ is indicated with subscripts when helpful. For $r, s \in \reat$, $\sect_{s\mlt, rh} = s^{2}r^{-1}\sect_{\mlt, h}$, where $s\mlt$ means the multiplication $sx\mlt y$.

Given an algebra $(\alg, \mlt)$, with underlying anticommutative and commutative products $[\dum, \dum]$ and $\sprod$, a straightforward computation shows that that $h$ is invariant with respect to both of $[\dum, \dum]$ and $\sprod$, consequently the sectional nonassociativities $\sect_{[\dum, \dum], h}(x, y)$ and $\sect_{\sprod, h}(x, y)$ are defined. From the identity
\begin{align}\label{prepoissonidentity}
\begin{aligned}
[x, y, z]_{\sprod} &+ [x, y, z]_{[\dum, \dum]} = [x, y, z]_{\sprod} + [[x, y], z] + [[y, z], x] \\
&= [x, y, z]_{\sprod} - [[z, x], y] + \at_{\mlt}(x, y, z) =2[x, y, z]_{\mlt} - 2[z, y, x]_{\mlt},
\end{aligned}
\end{align}
(subscripts indicate the dependence on the different multiplications) together with \eqref{hxxyy} there follows
\begin{align}
h([x, x, y]_{\sprod}, y) + h([x, x, y]_{[\dum, \dum]},y) = 4h([x, x, y]_{\mlt}, y),
\end{align}
and with \eqref{sectnadefined} this yields
\begin{align}
\begin{aligned}
4\sect_{\mlt, h}(x, y) & =  \sect_{\sprod, h}(x, y) + \sect_{[\dum, \dum], h}(x, y)\\
& =  \frac{h(x\sprod x , y\sprod y) - h(x\sprod y, y \sprod x)}{h(x, x)h(y, y)  - h(x, y)^{2}} +  \frac{h([x, y], [x, y])}{h(x, x)h(y, y)  - h(x, y)^{2}}.
\end{aligned}
\end{align}
It follows that for a Euclidean metrized algebra, $4\sect_{\mlt, h}(x, y) \geq \sect_{\sprod, h}(x, y)$ for all linearly independent $x$ and $y$, with equality if and only if $x$ and $y$ commute.

For $h \in S^{2}\alg^{\ast}$, define $h \kwedge h \in \mcurv(\alg)$ by
\begin{align}
&(h\kwedge h)(x, y, z, w) = h(x, z)h(y, w) - h(x, w)h(y, z), && \text{for all}\,\, x, y, z, w \in \alg.
\end{align}
By definition $(h\kwedge h)(x, y, x, y) = |x|^{2}|y|^{2} - h(x, y)^{2}$.

\begin{lemma}\label{constantsectcharlemma}
Let $\fie$ be an infinite field. A metrized $\fie$-algebra $(\alg, \mlt, h)$ of dimension $n \geq 2$ has constant sectional nonassociativity $c  \in \fie$ if and only if the quadratic forms determined on $\ext^{2}\alg$ by $-\tfrac{1}{2}(\acl^{\flat} + \acr^{\flat})$ and $c(h\kwedge h)$ coincide on decomposable elements of $\ext^{2}\alg$.
\end{lemma}

\begin{proof}
That the coincidence of $-\tfrac{1}{2}(\acl^{\flat} + \acr^{\flat})$ and $c(h\kwedge h)$ on decomposable elements of $\ext^{2}\alg$ implies constant sectional nonassociativity is immediate from the definition \eqref{sectnadefined}.

The Grassmannian $Gr(2, \alg)$ of two-dimensional subspaces of $\alg$ is a projective variety via the Plücker embedding $Gr(2, \alg) \to \proj(\ext^{2}\alg)$ sending $\fie\{x, y\}\in Gr(2, \alg)$ to the span $\fie\{x\wedge y\}$. 
The homogeneous functions $H, K:\ext^{2}\alg \to \fie $ defined by $H(x\wedge y) = h(x, x)h(y, y) - h(x, y)^{2}$ and $K(x\wedge y) = -\tfrac{1}{2}(\acl^{\flat} + \acr^{\flat})(x, y, x, y) = h(x\mlt x, y \mlt y) - h(x\mlt y, y\mlt x)$ are quadratic polynomials that restrict to the image of the Plücker embedding of the Grassmannian $Gr(2, \alg)$ of two-dimensional subspaces of $\alg$.
If $\fie$ is infinite, then the intersection of any two nonempty Zariski open subsets of $Gr(2, \alg)$ is nonempty. That the sectional nonassociativity be constant equal to $c$ is the same as assuming $c H - K$ vanishes on $Gr(2, \alg) \setminus \{H = 0\}$. As the complement $Gr(2, \alg) \setminus \{H = 0\}$ is infinite, if $c  H - K$ vanishes on $Gr(2, \alg) \setminus \{H = 0\}$, it vanishes on $Gr(2, \alg)$. 
\end{proof}

\begin{lemma}\label{constantsectlemma}
Let $(\alg, \mlt, h)$ be a metrized algebra over an infinite field of characteristic not equal to $2$ or $3$.
\begin{enumerate}
\item\label{csnfull} $(\alg, \mlt, h)$ has constant sectional nonassociativity $c$ if and only if $\P(\acl^{\flat} + \acr^{\flat}) = 2 c h \kwedge h$.
\item\label{csncyclic} If $(\alg, \mlt, h)$ is Lie-admissible, it has constant sectional nonassociativity $c$ if and only if $\acl^{\flat} + \acr^{\flat} = 2 c h \kwedge h$.
\item\label{csn} If $(\alg, \mlt, h)$ is Lie-admissible and flexible, it has constant sectional nonassociativity $c$ if and only if $\acl = \acr =  ch\kwedge h$.
\end{enumerate}
\end{lemma}

\begin{proof}
By definition of $\sect(x, y)$ and Lemma \ref{s2ext2lemma},
\begin{align}\label{sectexplicit}
\begin{split}
(h\kwedge h)(x, y, x, y)\sect(x, y) & = -\tfrac{1}{2}(\acl^{\flat} + \acr^{\flat})(x, y, x,y) = -\tfrac{1}{2}\P(\acl^{\flat} + \acr^{\flat})(x, y, x,y).
\end{split}
\end{align}
If there holds \eqref{csnfull}, then the last term of \eqref{sectexplicit} equals $ c(h\kwedge h)(x, y, x, y)$, so \eqref{sectexplicit} shows $\sect(x, y) = c$.  
On the other hand, by Lemma \ref{constantsectcharlemma}, that $\sect(x, y) = c$ for all linearly independent $x, y \in \alg^{\ast}$ means that the quadratic forms determined on $\ext^{2}\alg$ by $-\tfrac{1}{2}(\acl^{\flat} + \acr^{\flat})$ and $c(h\kwedge h)$ coincide on decomposable elements of $\ext^{2}\alg$. By Lemma \ref{s2ext2lemma} this implies \eqref{csnfull}. 
If $\mlt$ is Lie-admissible, then, by Lemma \ref{presectlemma}, $\P(\acl^{\flat} + \acr^{\flat}) = \acl^{\flat} + \acr^{\flat}$, so in this case \eqref{csnfull} is equivalent to \eqref{csncyclic}. If $\mlt$ is moreover flexible, then $\acl = \acr$, so \eqref{csncyclic} is equivalent to \eqref{csn}. 
\end{proof}

\begin{example}
Evidently the sectional nonassociativity of a metrized associative algebra is identically zero. More generally, the same is true for an alternative algebra. An algebra $(\alg, \mlt)$ is \emph{alternative} if $[x, x, y] = 0$ and $[x, y, y] = 0$ for all $x, y \in \alg$. By \eqref{sectnadefined}, the sectional nonassociativity of a metrized alternative algebra is identically zero. This can also be seen as a consequence of the theorem of E. Artin (see \cite[p. $18$]{Schafer-book}) that shows that the subalgebra generated by any two elements of an alternative algebra is associative.

For example, any Hurwitz algebra is alternative, so has vanishing sectional nonassociativity. 

This observation illustrates the limitations of the analogy between the associator and the curvature tensor, for the octonions are alternative but not associative. 
\end{example}

In the rest of this section it is supposed the base field is $\rea$. 

A Euclidean metrized algebra $(\alg, \mlt, h)$ has positive, nonnegative, zero, etc. sectional nonassociativity if the given qualifier is valid for every two-dimensional subspace of $\alg$. 

For a Euclidean metrized algebra, by the Cauchy-Schwarz inequality, the numbers
\begin{align}
&\bwu(\alg, \mlt, h) = \sup_{x, y \in \alg: x \wedge y \neq 0}\sect_{h}(x, y),&&\bwl(\alg, \mlt ,h) = \inf_{x, y \in \alg: x \wedge y \neq 0}\sect_{h}(x, y),
\end{align}
are finite and invariant under isometric automorphisms of $(\alg, \mlt, h)$. Their estimation for any well delineated class of metrized algebras is a basic problem.

Such an estimate is most interesting when the invariant metric $h$ is determined in some canonical manner by the algebra structure, as is the case for the Killing form $\tau_{\mlt}(x, y) = \tr L_{\mlt}(x)L_{\mlt}(y)$ of a Lie or Malcev algebra or the form $\tr L_{\mlt}(x\mlt y)$ on a Jordan algebra.

Lemma \ref{sectdirectsumlemma} shows that taking direct sums of Euclidean metrized commutative algebras preserves conditions such as nonpositive and nonnegative sectional nonassociativity.

For readability, in the proofs of Lemmas \ref{sectdirectsumlemma} and \ref{secttensorlemma}, there are written $|x|^{2} = h(x, x)$, $\lb x, y \ra = h(x, y)$. 

\begin{lemma}\label{sectdirectsumlemma}
Let $(\alg, \mlt, h)$ be a Euclidean metrized commutative algebra and suppose that $\alg = \alg_{1} \oplus \alg_{2}$ is a decomposition into $h$-orthogonal ideals. There hold:
\begin{align}\label{dssectminmax} 
&\bwl(\alg, \mlt, h) \geq \min\{\bwl(\alg_{1}, \mlt, h), \bwl(\alg_{2}, \mlt, h), 0\},&&
&\bwu(\alg, \mlt, h) \leq \max\{\bwu(\alg_{1}, \mlt, h), \bwu(\alg_{2}, \mlt, h), 0\}.
\end{align}
In particular, if $(\alg_{1}, \mlt, h)$ and $(\alg_{2}, \mlt, h)$ have nonpositive (resp. nonnegative) sectional nonassociativity, then $(\alg, \mlt, h)$ has nonpositive (resp. nonnegative) sectional nonassociativity.
\end{lemma}

\begin{proof}
For any $x_{1}, y_{1} \in \alg_{2}$ and $x_{2}, y_{2} \in \alg_{2}$, by the Cauchy-Schwarz inequality,
\begin{align}\label{splitsect2}
\begin{split}
|x\wedge y|^{2} &- |x_{1}\wedge y_{1}|^{2} - |x_{2}\wedge y_{2}|^{2} = |x_{1}|^{2}|y_{2}|^{2} + |x_{2}|^{2}|y_{1}|^{2} - 2\lb x_{1}, y_{1}\ra\lb x_{2}, y_{2}\ra =\\
&= \left(|x_{1}||y_{2}| - |x_{2}||y_{1}| \right)^{2}+ 2\left(|x_{1}||y_{1}| |x_{2}||y_{2}| -\lb x_{1}, y_{1}\ra \lb x_{2}, y_{2}\ra\right)\geq 0, 
\end{split}
\end{align}
with equality if and only if $x = x_{1} + x_{2}$ and $y = y_{1} + y_{2}$ are linearly dependent.
Let $x_{i} \in \alg_{i}$ be the $h$-orthogonal projection of $x \in \alg$. 
By \eqref{splitsect2}, $r_{i} = \tfrac{|x_{i}\wedge y_{i}|^{2}}{|x\wedge y|^{2}}$ is contained in $[0, 1]$, and, by the orthogonality of the ideals $\alg_{1}$ and $\alg_{2}$, it follows that for linearly independent $x, y \in \alg$,
\begin{align}\label{splitsect1}
\begin{aligned}
m_{1}r_{1} + m_{2}r_{2} \leq \sect_{\alg, h}(x, y) = \sect_{\alg_{1}, h}(x_{1}, y_{1})r_{1} + \sect_{\alg_{2}, h}(x_{2}, y_{2})r_{2} \leq M_{1}r_{1} + M_{2}r_{2}.
\end{aligned}
\end{align}
Applying to \eqref{splitsect1} the inequalities $\min\{a, b, c\} \leq ar_{1} + br_{2} + c(1 -r_{1} - r_{2}) \leq \max\{a, b, c\}$, valid for any $r_{1}, r_{2} \in [0, 1]$, with $(a, b, c) = (m_{1}, m_{2}, 0)$ or $(a, b, c) = (M_{1}, M_{2}, 0)$ yields \eqref{dssectminmax}.
\end{proof}

It is more difficult to relate the sectional nonassociativity of a tensor product to the sectional nonassociativities of its factors, but in some cases something can be said.

The tensor product of metrized algebras $(\alg_{i}, \mlt_{i}, h_{i})$, $i \in \{1, 2\}$, is the vector space $\alg_{1}\tensor \alg_{2}$ with the product defined by $(a_{1}\tensor b_{1}) \mlt (a_{2}\tensor b_{2}) = (a_{1}\mlt_{1}b_{1})\tensor (a_{1}\mlt_{2}b_{2})$ and extending bilinearly, and the metric defined by $h(a_{1}\tensor b_{1}, a_{2}\tensor b_{2}) = h_{1}(a_{1},b_{1})h_{2}(a_{1},b_{2})$ and extending bilinearly.

\begin{lemma}\label{secttensorlemma}
For $i \in \{1, 2\}$ let $(\alg_{i}, \mlt_{i}, h_{i})$ be Euclidean metrized commutative algebras. For $a_{i}, \bar{a}_{i} \in \alg_{i}$, if $\sect_{\alg_{1}, h_{1}}(a_{1}, \bar{a}_{1})$ and $\sect_{\alg_{2}, h_{2}}(a_{2}, \bar{a}_{2})$ are both nonnegative, then $\sect_{\alg_{1}\tensor \alg_{2}, h_{1}\tensor h_{2}}(a_{1}\tensor a_{2}, \bar{a}_{1}\tensor \bar{a}_{2})$ is nonnegative.
 \end{lemma}
 
 \begin{proof}
 Straightforward computations show
\begin{align}\label{productsect}
\begin{aligned}
|(a_{1}\tensor a_{2})\wedge (\bar{a}_{1}\tensor \bar{a}_{2})|^{2} &= |a_{1}\wedge \bar{a}_{1}|^{2}|a_{2}\wedge \bar{a}_{2}|^{2} + |a_{2}\wedge \bar{a}_{2}|^{2}\lb a_{1}, \bar{a}_{1}\ra^{2}  + |a_{1}\wedge \bar{a}_{1}|^{2}\lb a_{2}, \bar{a}_{2}\ra^{2},\\
\sect_{\alg_{1}\tensor \alg_{2}, h_{1}\tensor h_{2}}(a_{1}\tensor a_{2}, \bar{a}_{1}\tensor \bar{a}_{2}) &\tfrac{|(a_{1}\tensor a_{2})\wedge (\bar{a}_{1}\tensor \bar{a}_{2})|^{2}} {|a_{1}\wedge \bar{a}_{1}|^{2}|a_{2}\wedge \bar{a}_{2}|^{2}} = \sect_{\alg_{1}, h_{1}}(a_{1}, \bar{a}_{1})\sect_{\alg_{2}, h_{2}}(a_{2}, \bar{a}_{2})\\
& + \tfrac{|a_{1}\mlt_{1}\bar{a}_{1}|^{2}}{|a_{1}\wedge \bar{a}_{1}|^{2}}\sect_{\alg_{2}, h_{2}}(a_{2}, \bar{a}_{2}) +  \tfrac{|a_{2}\mlt_{2}\bar{a}_{2}|^{2}}{|a_{2}\wedge \bar{a}_{2}|^{2}}\sect_{\alg_{1}, h_{1}}(a_{1}, \bar{a}_{1}),
\end{aligned}
\end{align}
from which the claim follows.
\end{proof}

\section{Examples of algebras with constant sectional nonassociativity}\label{examplesection}
This section presents examples of metrized algebras having constant sectional nonassociativity. 

\begin{example}\label{2stepnaexample}
By Example \ref{2stepflatexample} a $2$-step nilpotent Lie algebra satisfies $\acl = \acr = 0$. If such a Lie algebra is equipped with an invariant metric, then its sectional nonassociativity vanishes identically. In general a $2$-step nilpotent Lie algebra need not admit an invariant metric. For example, the Heisenberg Lie algebra of any odd dimension admits no invariant metric. However, there are $2$-step nilpotent Lie algebras that do admit invariant metrics. G. Ovando \cite{Ovando-two-step} has given some examples. For example the semidirect product of a Lie algebra with its dual carries a split signature invariant metric, and if the original Lie algebra is $2$-step nilpotent, its semidirect product with its dual is $2$-step nilpotent as well. 

However, these examples obtained from $2$-step nilpotent real Lie algebras cannot be Euclidean - the invariance of the metric and the $2$-step nilpotence of $\g$ imply that $[\g, \g]$ is a nontrivial isotropic subspace.
\end{example}

\begin{example}\label{antiflexibleexample}
By Example \ref{antiflexibleasaexample} an antiflexible algebra has anti-self-adjoint-curvature, so  a metrized antiflexible algebra has vanishing sectional nonassociativity. The following example of a three-dimensional antiflexible algebra that is not power-associative is a modification (by unimportant scalar factors) of one given by F. Kosier in \cite[p. $303$]{Kosier}. Suppose $\chr \fie \neq 2$ and equip $\fie^{3}$ with the product $x\mlt y = (2x_{1}y_{1} + x_{2}y_{3}, 2x_{1}y_{2}, 2x_{3}y_{1})$ where $x = (x_{1}, x_{2}, x_{3})$ and $y = (y_{1}, y_{2}, y_{3})$. The two possible third powers of $(0, 1, 1)$ are different, so the algebra is not power-associative. 

On $(\fie^{3}, \mlt)$, the metric $h$ defined by $h(x, y) = x_{1}y_{1} + \tfrac{1}{2}(x_{2}y_{3} + x_{3}y_{2})$ is invariant, for $h(x\mlt y, z) = 2x_{1}y_{1}z_{1} + x_{2}y_{3}z_{1} + x_{1}y_{2}z_{3} + x_{3}y_{1}z_{2}$ is  preserved by cyclic permutations of $x$, $y$, and $z$. A computation shows $[x, y, z] = (0, 2x_{2}y_{3}z_{2}, - 2x_{3}y_{2}z_{3})$, which is evidently symmetric in $x$ and $z$. It follows that $h([x, x, y], y) =0$, and $(\alg, \mlt, h)$ has vanishing sectional nonassociativity. 

By Example \ref{antiflexibleadmissibleexample}, $(\alg, \mlt)$ is Lie admissible as can also be checked using that its underlying bracket is $[x, y] = (x_{2}y_{3} - x_{3}y_{2}, 2(x_{1}y_{2} - x_{2}y_{1}), 2(x_{3}y_{1} - x_{1}y_{3}))$. The underlying Lie algebra has Killing form equal to $8h$ and is isomorphic to $\sll(2, \fie)$ for the elements $p = (1, 0, 0)$, $q =(0, 1, 0)$, and $r = (0, 0, 1)$ are an $\sll(2)$-triple satisfying $[p, q] = r$, $[r, p] = 2p$, and $[r, q] = -2q$.

This example can be modified to yield a Euclidean example by taking $\so(3, \fie)$ in place of $\sll(2, \fie)$. Define a product $\star$ on $\fie^{3}$ by $x\star y = (x_{2}y_{3}, x_{3}y_{1}, x_{1}y_{2})$. The underlying anticommutative bracket is the usual vector cross product $x \times y = (x_{2}y_{3} - x_{3}y_{2}, x_{3}y_{1} - x_{1}y_{3}, x_{1}y_{2} - x_{2}y_{1})$ on $\fie^{3}$, while the underlying commutative product is its symmetric analogue $x \sprod y =(x_{2}y_{3} + x_{3}y_{2}, x_{3}y_{1} + x_{1}y_{3}, x_{1}y_{2} + x_{2}y_{1})$ (which is isomorphic to the $3$-dimensional case of the algebra of Example \ref{ealgexample}). That the products $\star$, $\times$, and $\sprod$ are invariant with respect to the metric $g(x, y) = x_{1}y_{1} + x_{2}y_{2} + x_{3}y_{3}$ follows from the cyclic symmetry of $g(x\star y, z) = x_{1}y_{2}z_{3} + x_{2}y_{3}z_{1} + x_{3}y_{1}z_{2}$ in $x$, $y$, and $z$. The associators satisfy 
\begin{align}
\begin{aligned}
[x, y, z]_{\star} &= ((x_{3}z_{3} - x_{2}z_{2})y_{1},(x_{1}z_{1} - x_{3}z_{3})y_{2} , (x_{2}z_{2} - x_{1}z_{1})y_{3}),\\
[x , y, z]_{\times} & = (x \times z) \times y = - [x, y, z]_{\sprod},
\end{aligned}
\end{align}
from which it is apparent that $\star$ is antiflexible and it follows that $g([x, x, y]_{\star}, y) = 0$, so that the Euclidean metrized algebra $(\fie^{3}, \star, g)$ has vanishing sectional nonassociativity, and $-g([x, x, y]_{\sprod}, y) = g([x, x, y]_{\times}, y) = g(x, x)g(y, y) - g(x, y)^{2}$ so that the Euclidean metrized algebras $(\fie^{3}, \sprod, g)$ and $(\fie^{3}, \times, g)$ have constant sectional nonassociativities $-1$ and $1$, respectively. The last statement is particular cases of Examples \ref{crossproductexample} and \ref{ealgexample}.
\end{example}

\begin{example}\label{crossproductexample}
A metrized algebra $(\alg, \mlt, h)$ is a (twofold) \emph{cross product algebra} if it satisfies:
\begin{enumerate}
\item\label{crossproduct1} The multiplication $\mlt$ is antisymmetric.
\item\label{crossproduct2} The quantity $h(x\mlt y, x\mlt y)$ equals the Gram determinant $h(x, x)h(y, y) - h(x, y)^{2}$.
\end{enumerate}
It follows from \eqref{crossproduct1} and the invariance of $h$ that the product $x\mlt y$ is $h$-orthogonal to $x$ and $y$, $h(x \mlt y, x) = 0$.
In \cite{Eckmann} and \cite{Brown-Gray} it is shown that a real cross product algebra exists only in dimensions $3$ and $7$, in which case it is determined uniquely up to isomorphism by the signature of $h$, the only possibilities being the positive definite case, and the case with negative inertial index one greater than the positive inertial index. 

\begin{theorem}
Let $\fie$ be an infinite field such that $\chr \fie \neq 2$. A metrized algebra $(\alg, \mlt, h)$ with antisymmetric multiplication is a cross product algebra if and only if it has constant sectional nonassociativity $1$.
\end{theorem}

\begin{proof}
The antisymmetry of $\mlt$ means that \eqref{sectnadefined} takes the form \eqref{acsect} It follows that a cross product algebra has constant sectional nonassociativity $1$. On the other hand, if a metrized algebra $(\alg, \mlt, h)$ with antisymmetric multiplication has constant sectional nonassociativity $1$ then $h(x\mlt y, x\mlt y)=h(x, x)h(y, y) - h(x, y)^{2}$ whenever $x$ and $y$ span an $h$-nondegenerate subspace. Because $\fie$ is infinite, arguing as in the proof of Lemma \ref{constantsectcharlemma}, this holds if and only if $h(x\mlt y, x\mlt y)=h(x, x)h(y, y) - h(x, y)^{2}$ for all $x$ and $y$, so that $(\alg, \mlt, h)$ is a cross product algebra.
\renewcommand{\qedsymbol}{\twoqedbox}
\end{proof}

\begin{corollary}
Suppose $\chr \fie \neq 2$. A metrized algebra $(\alg, \mlt, h)$ with anisotropic $h$ satisfies $h(x\mlt y, x\mlt y)=h(x, x)h(y, y) - h(x, y)^{2}$ for all $x, y \in \alg$ if and only if $(\alg, \mlt, h)$ is a cross product algebra.
\end{corollary}

\begin{proof}
To show the forward direction it suffices to show $\mlt$ is antisymmetric.
If $h(x\mlt y, x\mlt y)=h(x, x)h(y, y) - h(x, y)^{2}$ for all $x, y \in \alg$, then $h(x\mlt x, x\mlt x) = 0$ for all $x \in \alg$. Because $h$ is anisotropic, this implies $x\mlt x = 0$ for all $x \in \alg$. Because $\chr \fie \neq 2$, this implies $\mlt$ is antisymmetric. The reverse implication is immediate.
\renewcommand{\qedsymbol}{\twoqedbox}
\end{proof}

Hence a cross product algebra can be \emph{defined} as a metrized algebra with antisymmetric multiplication and having sectional nonassociativity $1$ on every $h$-nondegenerate subspace. Moreover, it is not necessary to suppose antisymmetry if the metric is anisotropic and $\chr \fie \neq 2$.
\renewcommand{\qedsymbol}{}
\end{example}

\begin{example}\label{ealgexample}
An example of a commutative algebra having constant negative sectional nonassociativity is the following. Consider $\rea^{n}$ with the coordinatewise product $x\cdot y = \sum_{i = 1}^{n}x_{i}y_{i}$. Define $\ell:\rea^{n} \to \rea$ by $\ell(x) = \sum_{i = 1}^{n}$ and define a modified product $x \mlt y = \tfrac{n+1}{n-1}x\cdot y - \tfrac{1}{n-1}(\ell(x)y + \ell(y)x)$. Then $(\rea^{n}, \mlt)$ is a commutative algebra metrized by  its Killing form 
\begin{align}
\tau_{\mlt}(x, y) = \tr L_{\mlt}(x)L_{\mlt}(y) = \tfrac{n+1}{n-1}\ell(x\cdot y) - \tfrac{1}{n-1}\ell(x)\ell(y), 
\end{align}
which is positive definite. With this metric, this algebra satisfies the reverse Norton inequality; in fact it has constant negative sectional nonassociativity. See \cite{Fox-simplicial} for details.
\end{example}

\begin{example}\label{3dexample}
This example describes a one-parameter family of pairwise nonisomorphic $3$-dimensional Euclidean metrized commutative algebras having constant sectional nonassociativity. 

For $\ep \in [0, \infty)$, define a commutative algebra $(\calg_{\ep}, \mlt)$ as the $3$-dimensional $\rea$-vector space $\calg_{\ep}$ with basis $\{f_{0}, f_{1}, f_{2}\}$ and multiplication $\mlt$ given by
\begin{align}\label{3depdefined}
\begin{aligned}
&f_{0} \mlt f_{0} = f_{0},& &f_{0}\mlt f_{1} = (\tfrac{1}{2} -\ep)f_{1}, & &f_{0}\mlt f_{2} = (\tfrac{1}{2}+\ep)f_{2},&\\
&f_{1}\mlt f_{1} = (\tfrac{1}{2}-\ep)f_{0},& &f_{1}\mlt f_{2} = 0,& & f_{2}\mlt f_{2} = (\tfrac{1}{2} + \ep)f_{0}.
\end{aligned}
\end{align}
For $x = x_{0}f_{0} + x_{1}f_{1} + x_{2}f_{2}$ the matrix of  $L_{\mlt}(x)$ with respect to the ordered basis $\{f_{0}, f_{1}, f_{2}\}$ is
\begin{align}\label{posexl}
 \begin{pmatrix}
x_{0} & (\tfrac{1}{2} -\ep)x_{1} & (\tfrac{1}{2} +\ep)x_{2}\\
(\tfrac{1}{2} -\ep)x_{1} & (\tfrac{1}{2} -\ep)x_{0}& 0\\
(\tfrac{1}{2} +\ep)x_{2} & 0 & (\tfrac{1}{2} +\ep)x_{0}
 \end{pmatrix},
\end{align}
A calculation using
\begin{align}\label{posexx}
\begin{aligned}
x\mlt y  =& (x_{0}y_{0} + (\tfrac{1}{2} - \ep)x_{1}y_{1} + (\tfrac{1}{2} + \ep)x_{2}y_{2})f_{0} + (\tfrac{1}{2}-\ep)(x_{0}y_{1} + x_{1}y_{0})f_{1} + (\tfrac{1}{2} + \ep)(x_{0}y_{2} + x_{2}y_{0})f_{2},
\end{aligned}
\end{align}
shows that the metric $h(x, y) = x_{0}y_{0} + x_{1}y_{1} + x_{2}y_{2}$ satisfies
\begin{align}
h(x\mlt y, z) = x_{0}y_{0}z_{0} + (\tfrac{1}{2} - \ep)(x_{1}y_{1}z_{0} + x_{1}y_{0}z_{1} + x_{0}y_{1}z_{1}) + (\tfrac{1}{2} + \ep)(x_{2}y_{2}z_{0} + x_{2}y_{0}z_{2} + x_{0}y_{2}z_{2}).
\end{align}
Because this is completely symmetric in $x$, $y$, and $z$ it shows $h$ is invariant with respect to $\mlt$. 
From \eqref{posexl} and \eqref{posexx} it follows that the matrix of $\ass_{\mlt}(x, y) = L_{\mlt}(x \mlt y) - L_{\mlt}(x)L_{\mlt}(y)$ is
\begin{align}\label{posexeic}
\begin{split}
 &(\tfrac{1}{4} - \ep^{2})\begin{pmatrix} x_{1}y_{1} + x_{2} y_{2} & -x_{0}y_{1} & -x_{0}y_{2}\\ -x_{1}y_{0} & x_{0}y_{0} + x_{2}y_{2} & -x_{1}y_{2} \\
-x_{2}y_{0} & -x_{2}y_{1} & x_{0}y_{0} + x_{1}y_{1}
\end{pmatrix}
=(\tfrac{1}{4} - \ep^{2})\left(h(x, y)\id_{\calg_{\ep}} - x \tensor h(y, \dum)\right),
\end{split}
\end{align}
Note that $\tr \ass_{\mlt}(x, y) = 2(1/4 - \ep^{2})h(x, y)$, so that $h$ is determined by $\mlt$. 
From \eqref{posexeic} it follows that $\sect(x, y) = \tfrac{1}{4} - \ep^{2}$ for all linearly independent $x$ and $y$ in $\calg_{\ep}$, so that $(\calg_{\ep}, \mlt, h)$ has constant sectional nonassociativity $\tfrac{1}{4} - \ep^{2}$. 

\begin{lemma}
If $\ep \neq 1/2$, the algebra $(\calg_{\ep}, \mlt)$ is simple.
\end{lemma}
\begin{proof}
By \eqref{posexl} and \eqref{posexx},
\begin{align}\label{detlx}
\det L_{\mlt}(x) = x_{0}(x_{0}^{2} + (\tfrac{1}{2} - \ep)x_{1}^{2} + (\tfrac{1}{2} + \ep)x_{2}^{2}) = \tfrac{1}{4}\tr L_{\mlt}(x) \tr L_{\mlt}(x\mlt x) .
\end{align}
If $\ideal \subset \calg_{\ep}$ is a nontrivial ideal so is its $h$-orthogonal complement $\ideal^{\perp}$, so without loss of generality it can be assumed $\ideal$ is one-dimensional and generated by $z \in \calg_{\ep}$. Consequently there is $\la \in \fie$ such that $z \mlt z = \la z$. Since $L_{\mlt}(z)\calg_{\ep} \subset \ideal$, $L_{\mlt}(z)$ has rank at most $1$ and $\tr L_{\mlt}(z) = \la$. 
By \eqref{detlx}, $0  = \det L_{\mlt}(z) = (\la/4)(\tr L_{\mlt}(z))^{2} = \la z_{0}^{2}$. Were $z_{0} \neq 0$, then $\la =0$ and $z \mlt z =0$. By \eqref{posexx}, $z \mlt z = 0$ if and only if $(1 - 2\ep)z_{1}z_{0} =0$, $(1+ 2\ep)z_{2}z_{0} = 0$, and $z_{0}^{2} + (1/2 - \ep)z_{1}^{2} + (1/2 + \ep)z_{2}^{2} = 0$. Because $2\ep \notin \{\pm 1\}$, if $z_{0} \neq 0$, the first two equations imply $z_{1} = 0$ and $z_{2} = 0$, which in the third equation yield the contradiction $z_{0}^{2} = 0$. 

Consequently, $z_{0} = 0$ and there are $a, b \in \rea$ such that $z = af_{1} + bf_{2}$ and $\la a f_{1} + \la bf_{2} = \la z = z \mlt z = ((\tfrac{1}{2} - \ep)a^{2}  + (\tfrac{1}{2} + \ep)b^{2})f_{0}$. Because $2\ep \notin \{\pm 1\}$, were $\la \neq 0$ this would imply $a = 0 = b$, a contradiction. If $\la = 0$ then $(\tfrac{1}{2} - \ep)a^{2}  + (\tfrac{1}{2} + \ep)b^{2} = 0$. If either of $a$ or $b$ is $0$ then so is the other. Were $ab \neq 0$, then it would follow from \eqref{posexl} that $L_{\mlt}(z) = L_{\mlt}(af_{1} + bf_{2})$ has rank $2$, contradicting that $L_{\mlt}(z)$ has rank at most $1$. 
\renewcommand{\qedsymbol}{\twoqedbox}
\end{proof}

\begin{lemma}\label{ceidemlemma}
If $\ep \neq 1/2$, then the nontrivial idempotents and square-zero elements in $(\calg_{\ep}, \mlt)$ have forms indicated in Table \ref{eicex}.
\begin{table}[!ht]
\renewcommand{\arraystretch}{1.3}
\begin{tabular}[t]{|c|c|c|c|c|}\hline
Element $x$ &Condition & Type&  $h(x, x)$ & Spectrum $L_{\mlt}(x)$\\[1ex]\hline
$f_{0}$ & none & Idempotent  &$1$ & $\{1, \tfrac{1}{2}(1 \pm 2\ep)\}$\\[1ex] \hline
$\tfrac{1}{1-2\ep}(f_{0} \pm \om f_{1})$ & $\om^{2} = \tfrac{4\ep}{2\ep - 1}$ & Idempotent &  $\tfrac{1-6\ep}{(1-2\ep)^{3}}$ &  $\{1, \tfrac{1}{2}\tfrac{1+2\ep}{1-2\ep}\}$ \\[1ex] \hline
$\tfrac{1}{1+2\ep}(f_{0} \pm \om f_{2})$ &  $\om^{2} = \tfrac{4\ep}{1 + 2\ep}$ & Idempotent & $\tfrac{1+6\ep}{(1+2\ep)^{3}}$ & $\{1, \tfrac{1}{2}\tfrac{1-2\ep}{1+2\ep}\}$ \\ [1ex]\hline
$\theta(f_{1} \pm \tfrac{\om}{1 + 2\ep}f_{2})$ & $\theta \in \fiet$, $\om^{2} = 4\ep^{2} - 1$ &Square-zero & $\tfrac{4\ep\theta^{2}}{1+2\ep}$ & $\{0, \pm \theta \si\}, \si^{2} = \ep(2\ep - 1)$\\[1ex] \hline
\end{tabular}
\caption{Idempotents and square-zero elements of $(\calg_{\ep}, \mlt)$ when $1\neq 1/2$.}\label{eicex}
\end{table}
\end{lemma}

\begin{proof}
From \eqref{posexx} it follows that $x \in \calg_{\ep}$ is an idempotent if and only if it solves
\begin{align}
\label{ceid1} x_{0} & = x_{0}^{2} + \tfrac{1}{2}(1-2\ep)x_{1}^{2} + \tfrac{1}{2}(1 + 2\ep)x_{2}^{2},\\
\label{ceid2} x_{1} &  = (1-2\ep)x_{0}x_{1},\\
\label{ceid3} x_{2} & = (1+2\ep)x_{0}x_{2}.
\end{align}
Combining \eqref{ceid2} and \eqref{ceid3} shows $(1-2\ep)x_{0}x_{1}x_{2} = x_{1}x_{2}= (1+2\ep)x_{0}x_{1}x_{2}$. Because $2\ep \neq \pm 1$ this implies $x_{0}x_{1}x_{2} = 0$. If $x_{0} = 0$, then, by \eqref{ceid2} and \eqref{ceid3}, there vanish both $x_{1}$ and $x_{2}$, so $x = 0$. Suppose $x_{0} \neq 0$. Then $x_{1}x_{2} = 0$. By \eqref{ceid1} it cannot be that both $x_{1}$ and $x_{2}$ vanish, so either $x_{2} = 0$ or $x_{1} = 0$. These two cases yield the idempotents listed in Table \ref{eicex}, provided there exist solutions to $\om^{2} = \tfrac{4\ep}{2\ep - 1}$ and $\om^{2} = \tfrac{4\ep}{1 + 2\ep}$ in $\fie$. 

From \eqref{posexx} it follows that $x \in \calg_{\ep}$ is square-zero if and only if it solves
\begin{align}
\label{cesz1} 0 & = x_{0}^{2} + \tfrac{1}{2}(1-2\ep)x_{1}^{2} + \tfrac{1}{2}(1 + 2\ep)x_{2}^{2},\\
\label{cesz2} 0 &  = (1-2\ep)x_{0}x_{1},\\
\label{cesz3} 0 & = (1+2\ep)x_{0}x_{2}.
\end{align}
Were $x_{0} \neq 0$, then, by \eqref{cesz2} and \eqref{cesz3}, there would vanish both $x_{1}$ and $x_{2}$, but in \eqref{cesz1} this yields a contradiction, so $x_{0} = 0$. Upon substituting $x_{0} = 0$ in \eqref{cesz1}, it follows that the square-zero elements are as indicated in the last line of Table \ref{eicex}. 
To see that $L_{\mlt}(x)$ is diagonalizable proceed as follows. If $(1 - 2\ep)a^{2} + (1+2\ep)b^{2} = 0$ admits a solution with $a, b \in \fie$, then necessarily $ab \neq 0$, and $4\ep^{2} - 1 = \left(\tfrac{(1+2\ep)b}{a}\right)^{2}$, so there is $\om \in \fie$ such that $4\ep^{2} - 1 = \om^{2}$ and $b = \pm\tfrac{\om}{1 + 2\ep}a$. In this case write $z_{\pm} = f_{1} \pm \tfrac{\om}{1 + 2\ep}f_{2}$. It is straightforward to check that $z_{+}\mlt z_{-} = (1 - 2\ep)f_{0}$, $z_{+}\mlt f_{0} = \tfrac{1}{2}z_{+} - \ep z_{-}$, and $z_{-}\mlt f_{0} = -\ep z_{+} + \tfrac{1}{2}z_{-}$. If there is $\si \in \fie$ such that $\si^{2} = \ep(2\ep - 1)$, then $\pm \si f_{0} + \tfrac{1}{2}z_{+} - \ep z_{-}$ are eigenvectors of $L_{\mlt}(z_{+})$ with eigenvalues $\pm \si$ and $\pm \si f_{0} -\ep z_{+} + \tfrac{1}{2} z_{-}$ are eigenvectors of $L_{\mlt}(z_{-})$ with eigenvalues $\pm \si$.
\renewcommand{\qedsymbol}{\twoqedbox}
\end{proof}

\begin{corollary}
For $\ep, \bar{\ep} \in [0, \infty)$, $(\calg_{\ep}, \mlt)$ and $(\calg_{\bar{\ep}}, \mlt)$ are not isomorphic if $\ep \neq \bar{\ep}$. 
\end{corollary}

\begin{proof}
The algebra $(\calg_{1/2}, \mlt)$ is associative, isomorphic to a direct sum of a trivial ideal spanned by $f_{1}$ and a two-dimensional ideal spanned by $f_{0}$ and $f_{2}$ that is isomorphic to the quadratic algebra $\rea[t]/(t^{2} - 1)$.
When $\ep \neq 1/2$, $(\calg_{\ep}, \mlt)$ is not power associative. For example, for $x= f_{1} + f_{2}$, $((x\mlt x)\mlt x)\mlt x = (\tfrac{1}{2} +2 \ep^{2})f_{0}$ while $((x\mlt x)\mlt (x \mlt x)) = f_{0}$. Consequently, for $\ep \neq 1/2$, $(\calg_{\ep}, \mlt)$ is not isomorphic to the associative algebra $(\calg_{1/2}, \mlt)$. 

An algebra isomorphism maps idempotents to idempotents and preserves the spectra of the corresponding left multiplication endomorphisms. By Lemma \ref{ceidemlemma}, the algebra $(\calg_{0}, \mlt)$ has a unique nontrivial idempotent and only one one-dimensional subalgebra, while for all $\ep > 0$, $(\calg_{\ep}, \mlt)$ has at least three distinct idempotents, so, for $\ep > 0$,  $(\calg_{\ep}, \mlt)$ is not isomorphic to $(\calg_{0}, \mlt))$. 

Suppose $\ep, \bar{\ep} \in (0, 1/2)\cup (1/2, \infty)$ and $\phi:\calg_{\ep} \to \calg_{\bar{\ep}}$ is an algebra isomorphism. By Lemma \ref{ceidemlemma}, for $\om^{2} = \tfrac{4\ep}{1+2\ep}$, $\tfrac{1}{1+2\ep}\phi(f_{0}  \pm \om f_{2}) \in \idem(\calg_{\bar{\ep}}, \mlt)$ whose left multiplication endomorphisms have the same spectrum $\{1, 1/2, \tfrac{1-2\ep}{1+2\ep}\}$. Because $\tfrac{1-2\ep}{1+2\ep} \neq 1$, it follows from Lemma \ref{ceidemlemma} that either $\tfrac{1-2\ep}{1+2\ep} = \tfrac{1-2\bar{\ep}}{1+2\bar{\ep}}$ or $\tfrac{1-2\ep}{1+2\ep} = \tfrac{1+2\bar{\ep}}{1-2\bar{\ep}}$. Because $f(t) = \tfrac{1-2t}{1+2t}$ decreases monotonically on $[0, \infty)$, the first case implies $\bar{\ep} = \ep$. The second case implies $\ep + \bar{\ep} =0$, which is impossible because $\ep$ and $\bar{\ep}$ are both positive.
\renewcommand{\qedsymbol}{\twoqedbox}
\end{proof}

In summary, $(\calg_{\pm \ep}, \mlt, h)$ is a one-parameter family of  pairwise nonisomorphic Euclidean commutative metrized algebras having strictly positive constant sectional nonassociativity for $\ep \in [0, 1/2)$ and strictly negative constant sectional nonassociativity for $\ep \in (1/2, \infty)$ and which are simple when $\ep \neq 1/2$.
\renewcommand{\qedsymbol}{}
\end{example}

\section{Examples of bounds on sectional nonassociativity}\label{boundssection}

A Euclidean metrized algebra $(\alg, \mlt, h)$ has nonnegative sectional nonassociativity if and only if 
\begin{align}\label{norton}
&0 \leq h([x, x, y], y) = h(x\mlt x, y \mlt y) - h(x \mlt y, y \mlt x),&& \text{for all}\,\, x, y \in \alg.
\end{align}
That is, the associator endomorphism $\ass_{\mlt}(x, x) = [x, x, \dum]$ is nonnegative definite for all $x \in \alg$.
For metrized commutative algebras \eqref{norton} is known as the \emph{Norton inequality} because S. P. Norton showed that it holds for the Griess algebra of the monster finite simple group \cite{Conway-monster, Norton}. 

\begin{example}\label{voaexample}
A \emph{vertex operator algebra} (VOA) over a field $\fie$ of characteristic zero is a $\fie$-vector space $\ste$ equipped with a linear \emph{state-field correspondence} $Y: \ste \to (\eno\ste)[[z, z^{-1}]]$ taking values in endomorphism-valued formal power series and written $Y(a, z) = \sum_{n \in \integer}a_{(n)}z^{-n-1} \in (\eno\ste)[[z, z^{-1}]]$, and two distinguished vectors, the \emph{vacuum vector} $\one \in \ste$ and the \emph{conformal vector} $\cc \in \ste$, that together satisfy some axioms that can be found in \cite{Frenkel-Ben-zvi, Frenkel-Lepowsky-Meurman, Kac-vertexbook}.

The axioms include that the endomorphisms $L_{m} = \cc_{(m+1)}$ satisfy $[L_{m}, L_{n}] = (m-n)L_{m+n} + \tfrac{m^{3} - m}{12}c\delta_{m+ n, 0}$ for some $c \in \fie$ called the \emph{central charge} of the VOA, and that the operator $L_{0} = \cc_{(1)}$ is semisimple and there is a direct sum decomposition $\ste = \oplus_{n = 0}^{\infty}\ste^{n}$ where $\ste^{n}$ is the eigenspace of $L_{0}$ with eigenvalue $n$, which is moreover finite-dimensional for all $n \geq 0$.
These force $\one \in \ste^{0}$ and $\ste^{i}_{(n)}\ste^{j} \subset \ste^{i+j-n-1}$. 

A VOA is OZ (short for \emph{one-zero}) if $\ste^{0}$ is generated by $\one$ and $\ste^{1}= \{0\}$. 
For an OZ VOA, the multiplication $\star$ defined by $a \star b = a_{(1)}b$ for $a,b \in \ste^{2}$ makes $\ste^{2}$ into a commutative algebra and the symmetric bilinear form $g$ defined by $g(a, b)\one = a_{(3)}b$ for $a, b \in \ste^{2}$ is invariant with respect to $\star$. This is explained in detail in \cite[section $3.5$]{Gebert} and \cite[Section $8$]{Matsuo-Nagatomo}. 
By definition, $L_{\star}(\cc) = 2\Id_{\alg}$ and $g(\cc, \cc) = c/2$. Hence $e = \cc/2$ is a unit in $ \ste^{2}$ such that $g(e, e) = c/8$. The triple $(\ste^{2}, \star, g)$ is called the \emph{Griess algebra} of the OZ VOA. 

By a theorem of Miyamoto \cite[Theorem $6.3$]{Miyamoto-griessalgebras}, the Griess algebra of a real OZ VOA having a positive definite invariant bilinear form satisfies the Norton inequality, so has nonnegative sectional nonassociativity. 

Because, as mentioned in the introduction, the simple Euclidean Jordan algebras can be obtained as Griess algebras of VOAs, it follows from the theorem of Miyamoto that they have nonnegative sectional nonassociativity. Theorem \ref{hermsecttheorem} indicates a direct proof (it is not hard) and the characterization of equality.
\end{example}

Nonpositive sectional nonassociativity is equivalent to the \emph{reverse Norton inequality}, which is \eqref{norton} with the inequality reversed. Example \ref{ealgexample} gives an example of algebras satisfying the reverse Norton inequality. 

Example \ref{compositionsectexample} shows the sectional nonassociativities of symmetric composition algebras can take both signs and shows sharp upper and lower bounds for them.

\begin{example}\label{compositionsectexample}
A composition algebra is \emph{symmetric} if the symmetric bilinear form $h(x, y) = q(x + y) - q(x) - q(y)$ is invariant. Basic results about symmetric composition algebras are given in \cite[Chapter $34$]{Knus-Merkurjev-Rost-Tignol}. Any Hurwitz algebra $(\alg, \mlt, q, e)$ has an associated \emph{para-Hurwitz} algebra $(\alg, \pmlt, q)$ that is a symmetric composition algebra. Its multiplication is the isotope of $\mlt$ determined by the canonical involution and so defined by $x\pmlt y = \bar{x}\mlt \bar{y}$. It is not unital, but the unit $e$ of the original Hurwitz algebra is a distinguished idempotent for which $L_{\pmlt}(e) = R_{\pmlt}(e)$ is the canonical involution of the original Hurwitz algebra.

Symmetric composition algebras have been classified by A. Elduque and H.~C. Myung in \cite{Elduque-composition, Elduque-Myung, Elduque-Myung-composition, Knus-Merkurjev-Rost-Tignol}; not quite all examples are given by the para-Hurwitz and forms of what are known as Okubo or Petersen algebras. In particular, a symmetric composition algebra has dimension in $\{1, 2, 4, 8\}$, it is alternative, it is nonunital if $\dim \alg > 1$, and it is commutative if and only if $\dim \alg \leq 2$. 

\begin{lemma}\label{symmetriccompositionlemma}
Let $(\alg, \mlt, q)$ be a real symmetric composition algebra having associated bilinear form $h$. If $h$ is positive definite, the sectional nonassociativity of $(\alg, \mlt, h)$ satisfies $|\sect(x, y)| \leq 1$ for all linearly independent $x, y \in \alg$. There holds $\sect(x, y) = 1$ if and only if $x \mlt x$ and $y \mlt y$ are linearly dependent, and there holds $\sect(x, y) = -1$ if and only if $x \mlt y = y \mlt x$.
\end{lemma}
\begin{proof}
Linearizing the relation $q(x \mlt y) = q(x)q(y)$ shows that $h$ satisfies
\begin{align}\label{compositionlinearized}
h(x \mlt y, w \mlt z) + h(w \mlt y, x \mlt z) = h(x, w)h(y,z ) = h(y \mlt x, z \mlt w) + h(z \mlt x, y \mlt w).
\end{align}
Taking $z = x$ and $w = y$ in \eqref{compositionlinearized} yields 
\begin{align}\label{compositionidentity}
h(x \mlt y, y \mlt x) + h(x \mlt x, y \mlt y) = h(x, y)^{2}, 
\end{align}
while the multiplicativity of $q$ yields $2|x\mlt y|^{2} = |x|^{2}|y|^{2} = 2|y \mlt x|^{2}$. Together these yield
\begin{align}
|x|^{2}|y|^{2} - h(x, y)^2 = 2|x \mlt y|^{2} - h(x\mlt y, y \mlt x) - h(x \mlt x, y\mlt y).
\end{align}
There results
\begin{align}\label{compositionsect1}
\begin{split}
\sect(x, y) & = \tfrac{ h(x \mlt x, y\mlt y) - h(x\mlt y, y \mlt x) }{|x|^{2}|y|^{2} - h(x, y)^2} = -1 + \tfrac{2\left(|x \mlt y|^{2} - h(x\mlt y, y \mlt x) \right)}{|x|^{2}|y|^{2} - h(x, y)^2} = -1 + \tfrac{ |x \mlt y - y \mlt x|^{2}}{|x|^{2}|y|^{2} - h(x, y)^2} \geq -1,
\end{split}
\end{align}
where the second equality follows from $|x\mlt y| = |y \mlt x|$. Equality holds if and only if $x \mlt y = y \mlt x$. On the other hand, \eqref{compositionidentity} yields
\begin{align}\label{compositionsect2}
\begin{split}
\sect(x, y) 
& = \tfrac{ h(x \mlt x, y\mlt y) - h(x\mlt y, y \mlt x) }{|x|^{2}|y|^{2} - h(x, y)^2} 
=  \tfrac{ 2h(x \mlt x, y\mlt y) - h(x, y)^{2} }{|x|^{2}|y|^{2} - h(x, y)^2}  \leq  \tfrac{ 2|x\mlt x||y\mlt y|- h(x, y)^{2} }{|x|^{2}|y|^{2} - h(x, y)^2} = 1, 
\end{split}
\end{align}
the inequality by the Cauchy-Schwarz inequality, and the last equality because $2|x\mlt x|^{2} = |x|^{4}$. Equality holds if and only if $x \mlt x$ and $y \mlt y$ are linearly dependent.
\renewcommand{\qedsymbol}{\twoqedbox}
\end{proof}

Lemma \ref{symmetriccompositionlemma} implies that a para-Hurwitz algebra has sectional nonassociativities of both signs provided its dimension is at least $4$. For a para-Hurwitz algebra $(\alg, \pmlt, q)$ with underlying Hurwitz algebra $(\alg, \mlt, q, e)$, there holds $h(x, e) =0$ if and only if $x\pmlt x = -\tfrac{1}{2}h(x, x)e$ and $x$ is not a multiple of $e$. Consequently, by Lemma \ref{symmetriccompositionlemma}, there holds $\sect_{h, \pmlt}(e, x) = -1$ if $h(x, e) = 0$ and $\sect_{h, \pmlt}(x, y) = 1$ if $h(x, e) = 0$, $h(y, e) = 0$, and $x$ and $y$ are linearly independent. The latter case cannot occur for a two-dimensional para-Hurwitz algebra, but does occur for para-quaternions or para-octonions. On the other hand, it follows that a two-dimensional para-Hurwitz algebra has constant sectional nonassociativity $-1$.

The complex Okubo algebra is $\balg = \sll(3, \com)$ equipped with the multiplication
\begin{align}\label{okuboproduct}
\begin{aligned}
x \star y &= \om xy - \om^{2}yx + \tfrac{\om - \om^{2}}{3}h(x, y)I = -\tfrac{1}{2}[x, y] + \tfrac{\om - \om^{2}}{2}\left(xy + yx + \tfrac{2}{3}h(x, y)I\right).
\end{aligned}
\end{align}
where juxtaposition indicates the ordinary matrix product, $[x, y] = xy - yx$ is the usual matrix commutator, $h(x, y) = -\tr(xy)$, and $\om$ is a nontrivial cube root of unity (so $\om + \om^{2} = -1$). The product \eqref{okuboproduct} or an equivalent product was studied independently by Okubo \cite{Okubo-pseudo, Okubo-Myung-division, Okubo-Osborn}, Petersson \cite{Petersson-funften}, and by Laquer \cite{Laquer} (as one of a family of products on $\su(n)$).

From \eqref{okuboproduct} it is apparent that $[x, y]_{\star} = -[x, y]$ so that $(\balg, \star)$ is Lie-admissible, with underlying Lie algebra isomorphic to $\sll(3, \com)$ with the opposite of the usual Lie bracket. The underlying symmetric product $x\sprod y =  (\om - \om^{2})(xy + yx + \tfrac{2}{3}h(x, y)I)$ is a rescaling of the usual trace-free Jordan product.

Straightforward computations show $[x, y, z]_{\star} = [y, [z, x]] - h(x, y)z + h(z, y)x$, so that $(x\star y)\star x = \tfrac{1}{2}h(x, x)y = x\star(y\star x)$, showing $(\balg, \star)$ is a symmetric composition algebra. Its sectional nonassociativity satisfies 
\begin{align}\label{okubosect}
\sect_{\balg, \star, h}(x, y) = \tfrac{h([x, x, y]_{\star}, y) }{h(x, x)h(y, y) - h(x, y)^{2}} = \tfrac{h([x, y], [x, y])}{h(x, x)h(y, y) - h(x, y)^{2}} - 1.
\end{align}
The real linear map $\si \in \eno_{\rea}(\balg)$ defined by $\si(x) = -\bar{x}^{t}$ is an order two isometric automorphism of $(\balg, \star)$. Its fixed point set is $\alg  = \su(3)$, and so $(\alg, \star)$ is a subalgebra metrized by the Euclidean inner product $h(x, y) = \tfrac{1}{2}\tr(\bar{x}^{t}y + \bar{y}^{t}x)$. The algebra $(\alg, \star, h)$ is the compact real form of the Okubo algebra. From \eqref{okubosect} it follows that $-1 \leq \sect_{\alg, \star, h}(x, y)$ with equality if and only if $[x, y] = 0$. As explained in Remark \ref{killingbwremark}, from the Böttcher-Wenzel-Chern-do Carmo-Kobayashi inequality, Theorem \ref{bwtheorem}, it follows that $|[x, y]|^{2} \leq 2|x|^{2}|y|^{2}$ for $x, y \in \su(3)$, so, from \eqref{okubosect} it follows that $1 \geq \sect_{\alg, \star, h}(x, y)$; the case of equality is characterized in Theorem \ref{bwtheorem}.
\renewcommand{\qedsymbol}{}
\end{example}

The remainder of this section is devoted to the formulation and proof of Theorem \ref{hermsecttheorem}, which gives sharp bounds on sectional nonassociativity of simple Euclidean Jordan algebras.

Let $\mat(n, \hur)$ denote the $n \times n$ matrices over the Hurwitz algebra $\hur$. The subspace of $n \times n$ Hermitian matrices over $\hur$ is the subspace $\herm(n, \hur) = \{x \in \mat(n, \hur): \bar{x}^{t} = x\}$ comprising the fixed points of the conjugate transpose $x \to \bar{x}^{t}$. The space $\herm(n, \hur)$ is a commutative algebra with unit the identity matrix, $I$, when equipped with the symmetrized matrix product $x \star y = \tfrac{1}{2}(xy + yx)$, where juxtaposition, as in the expression $xy$, indicates the matrix product in $\herm(n, \hur)$. The algebra $(\herm(n, \hur), \star)$ is Jordan if $\hur$ is associative or $\hur = \cayley$ and $n = 3$ \cite[Section III.$1$, Theorem $1$]{Jacobson-jordan}. The algebra $(\herm(2, \hur), \star)$ is isomorphic to the $(2 + \dim \hur)$-dimensional Jordan algebra of a quadratic form often called a \emph{spin factor}; this follows from the observation that the $\star$-square of a trace-free element is a multiple of the unit. 

The matrix trace $\tr:\mat(n, \hur) \to \hur$ is a $\fie$-linear map satisfying $\tr(\bar{x}^{t}) =\tr(\bar{x}) =\overline{\tr(x)}$ and so also $\nt(\tr(x)) = \tr(x + \bar{x})$. The $\fie$-bilinear form $h$ on $\mat(n, \hur)$ defined by $h(x, y) = \tfrac{1}{2n}\nt(\tr(\bar{x}^{t}y))$ is symmetric (even when $\hur$ is noncommutative). 
Its restriction to $\herm(n, \hur)$ satisfies $h(x, y) = \tfrac{1}{2n}\nt(\tr(xy)) = \tfrac{1}{n}\nt(\tr(x\star y))$. The dimension dependent constant factor is a normalization. 

The \emph{nucleus} $\nucleus(\alg, \mlt)$ of an algebra $(\alg, \mlt)$ is the set of elements $a \in \alg$ such that there vanish the associators $[a, x, y]$, $[x, a, y]$, and $[y, a, x]$ for all $x, y \in \alg$. It is straightforward to see that if $(\alg, \mlt)$ is commutative, then $\nucleus(\alg, \mlt) = \{a \in \alg: [L_{\mlt}(a), L_{\mlt}(x)] =0\,\,\text{for all}\,\, x \in \alg\}$.
If a unital Jordan algebra $(\alg, \mlt)$ has nondegenerate trace form $\tr L_{\mlt}(x\mlt y)$ and $\si \in S^{2}\alg^{\ast}$ is an invariant symmetric bilinear form, then there is $z \in \nucleus(\alg, \mlt)$ such that $\si(x, y) = \tr L_{\mlt}((z\mlt x)\mlt y)$ for all $x, y \in \alg$ \cite[Chapter $3$, Theorem $10$]{Koecher}. 

\begin{lemma}\label{hermitianformslemma}
Suppose $\chr \fie \notdivides 2n$. Let $\hur$ be a $d$-dimensional Hurwitz $\fie$-algebra and let $N = n + dn(n-1)/2$. The symmetric bilinear form $h$ on $\herm(n, \hur)$ defined by $h(x, y) = \tfrac{1}{2n}\nt(\tr(xy)) = \tfrac{1}{n}\nt(\tr(x\star y))$ is nondegenerate and invariant and satisfies $\tr L_{\star}(x, y) = Nh(x, y)$ for all $x, y \in \herm(n, \hur)$.
\end{lemma}
\begin{proof}
A consequence of the involutive invariance of the metric $h$ on $\hur$ is that the linear map $\nt:\hur \to \rea$ defined by $\nt(x) = x + \bar{x}$ satisfies 
\begin{align}\label{realinvar}
&\nt((xy)z) = h((xy)z, e) = h(xy, \bar{z}) = h(x, \bar{z}\bar{y}) = h(x(yz), e) = \nt(x(yz)), && \text{for}\,\, x, y, z \in \hur.
\end{align}
From \eqref{realinvar} it follows that $\nt(\tr((xy)z)) = \tr((xy)z + \bar{z}(\bar{y}\bar{x})) = \tr(x(yz) +(\bar{z}\bar{y})\bar{x}) = \nt(\tr(x(yz))$ for $x, y, z \in \mat(n, \hur)$, and there follows
\begin{align}
\begin{split}
2nh(x\star y, z) &= 2\nt(\tr(x\star y)z) = \nt(\tr((xy)z)) + \nt(\tr((yx)z)) \\
&= \nt(\tr(x(yz))) + \nt(\tr(z(yx))= \nt(\tr(x(yz))) + \nt(\tr((zy)x)) \\
&= \nt(\tr(x(yz))) + \nt(\tr(x(zy))) = 2\nt(\tr(x(y\star z)) = 2nh(x, y\star z),
\end{split}
\end{align}
showing $h$ is invariant on $\herm(n, \hur)$. Using a basis of $\herm(n, \hur)$ as in \cite[pp. $125-126$]{Jacobson-jordan} is straightforward to check that $h$ is nondegenerate.

From the invariance of $h$ is follows that there is $z \in \nucleus(\herm(n, \hur), \star)$ such that $h(x, y) = \tr L_{\star}((z\star x)\star y)$ for all $x, y \in \herm(n, \hur)$. The \emph{center} of a Jordan algebra is the subset of its nucleus comprising elements that commute with all other elements. By a theorem of A.~A. Albert \cite{Albert-nuclei}, the nucleus of a simple Jordan $\fie$-algebra ($\chr \fie \neq 2$) equals its center. Consequently, since $(\herm(n, \hur), \star)$ is simple, its nucleus equals its center. Because $(\herm(n, \hur), \star)$ is central simple, its center is its subalgebra, isomorphic to $\fie$, spanned by its unit \cite[Chapter $5$, Section $7$]{Jacobson-jordan}. It follows that $z$ is a multiple of the unit, so $h$ is a multiple of $\tr L_{\star}(\dum \star \dum)$. Because $h(e, e) = 1$ and $\tr L_{\star}(e\star e) = \dim \herm(n, \hur) = N = n + dn(n-1)/2$ where $d = \dim \hur$, it follows that $\tr L_{\star}(\dum \star \dum) = Nh$.
\end{proof}

The simple Euclidean Jordan algebras are classified; a modern exposition is \cite{Faraut-Koranyi}. A simple Euclidean Jordan algebra of rank $n \geq 4$ is isomorphic to $\herm(n, \hur)$ for $\hur$ an associative Hurwitz algebra and a simple Euclidean Jordan algebra of rank $n = 3$ is isomorphic to $\herm(n, \hur)$ for $\hur$ a Hurwitz algebra. The rank $2$ case consists of the spin factor algebras; as mentioned these include $\herm(2, \hur)$.

In the statement of Theorem \ref{hermsecttheorem}, $e_{ij}$ denotes the matrix with a $1$ in the $ij$ entry and $0$ in other entries.

\begin{theorem}\label{hermsecttheorem}
Let $\hur$ be a real Hurwitz algebra. 
Equip $(\herm(n, \hur), \star)$ with the invariant metric $h(x, y) = N^{-1}\tr L_{\star}(x\star y)= n^{-1}\tr(x\star y)$ where $N = \dim \herm(n, \hur)$ and $n \geq 2$ is the rank of $\herm(n, \hur)$. For linearly independent $x, y \in \herm(n, \hur)$, the sectional nonassociativity of $(\herm(n, \hur), \star, h)$ satisfies
\begin{align}\label{hermsect}
0 \leq \sect(x, y) = \sect_{h, \star}(x, y) \leq \tfrac{n}{2},
\end{align}
for all linearly independent $x, y \in \herm(n, \hur)$.
\begin{enumerate}
\item If $\hur$ is associative, equality holds in the lower bound of \eqref{hermsect} if and only if $x$ and $y$ commute with respect to the matrix product; in particular, if $x$ and $x \star x$ are linearly independent, then $\sect(x, x\star x) = 0$. 
\item Equality holds in the upper bound of \eqref{hermsect} if and only if $x$ and $y$ are simultaneously equivalent under $\Aut(\herm(n, \hur), \star)$ to scalar multiples of $e_{11} - e_{nn}$ and $e_{1n} + e_{n1}$.
\end{enumerate}
\end{theorem}

\begin{proof}
For $x, y \in \mat(n, \hur)$, let $[x, y] = xy - yx$. By \eqref{prepoissonidentity}, the associators 
of $\star$ and the matrix product $\cdot$ satisfy 
\begin{align}\label{jordanassociators}
\begin{split}
4[x, y, z]_{\star}
& = 2[x, y, z]_{\cdot} - 2[z, y, x]_{\cdot} + [x, [y, z]] + [z, [x, y]].
\end{split}
\end{align}
If $\hur$ is associative, $(\mat(n, \hur), \cdot)$ is associative, so \eqref{jordanassociators} becomes simply $4[x, y, z]_{\star} = [y, [x, z]]$. 
Hence
\begin{align}\label{hermsect0}
\begin{split}
h(x\star x, z \star z) & - h(x\star z, z \star x)  = h([x, x, z]_{\star}, z) = \tfrac{1}{4}h([x, [x, z]], z) = \tfrac{1}{4n}\tr([x, [x, z]]\star y) \\
&= \tfrac{1}{8n}\tr\left([x, [x, z]]z + z[x, [x, z]]\right) = -\tfrac{1}{4n}\tr [x, z][x, z]= \tfrac{1}{4n}\tr \overline{[x, z]}^{t}[x, z].
\end{split}
\end{align}
When $\hur$ is associative, an automorphism of $(\herm(n, \hur), \star)$ preserves the ordinary matrix product and commutator bracket on $\mat(n, \hur)$, but I do not know if this is true when $\hur = \cayley$, so in this case the following alternative argument is given (it works for associative $\hur$ too).

Recall that the principal axis theorem states that every element of $\herm(n, \hur)$ (where $n\leq 3$ if $\hur = \cayley$) is equivalent via an element of $\Aut(\herm(n, \hur), \star)$ to a diagonal matrix. For $\hur \in \{\rea, \com, \quat\}$ this is well known. For $\hur = \cayley$ this is \cite[Theorem $5.1$]{Freudenthal}; see also \cite[Theorem V.$2.5$]{Faraut-Koranyi}.

By the principal axis theorem there are $g \in \Aut(\herm(n, \hur), \star))$ and diagonal $\la \in \herm(n, \hur)$ such that $x = g\la$. Let $y = g^{-1}z$. Since $\la$ is Hermitian its elements are real, so $L_{\cdot}(\la \cdot \la) = L_{\cdot}(\la)^{2}$ and $R_{\cdot}(\la \cdot \la) = R_{\cdot}(\la)^{2}$, and hence $[\la, \la, y]_{\cdot} =  [y, \la, \la]_{\cdot} = 0$. In \eqref{jordanassociators} this yields
\begin{align}
 4[x, x, z]_{\star} = 4g[\la, \la,y]_{\star} = g[\la, [\la, y]] .
\end{align} 
As in \eqref{hermsect0} this yields
\begin{align}\label{hermsect0b}
\begin{split}
h&(x\star x, z \star z) - h(x\star z, z \star x)  = h([x, x, z]_{\star}, z) = \tfrac{1}{4}h(g[\la,  [\la, y]], gy)  = \tfrac{1}{4}h([\la,  [\la, y]], y)\\
&= \tfrac{1}{4n}\tr([\la, [\la, y]]\star y) = \tfrac{1}{8n}\tr\left([\la, [\la, y]]y + y[\la, [\la, y]]\right) = -\tfrac{1}{4n}\tr [\la, y][\la, y]= \tfrac{1}{4n}\tr \overline{[\la, y]}^{t}[\la, y],
\end{split}
\end{align}
where, in the case $\hur = \cayley$, the penultimate equality follows from the invariance $\tr([x, y]z) + \tr(y[x, z]) = 0$ for $x, y, z \in\herm(n, \hur)$ proved in \cite[section $4.4$]{Freudenthal}. 
By \eqref{hermsect0} and \eqref{hermsect0b}, $\sect(x, y) \geq 0$. If $\hur$ is associative, by \eqref{hermsect0}, equality holds if and only if $[x, y] = 0$; in particular, $[x\star x, x] = [x, x, x]_{\star} = 0$ because $\star$ is commutative, so if $x$ and $x \star x$ are linearly independent, then $\sect(x, x \star x) = 0$.

Since $\tr\bar{x}^{t}x$ equals the Frobenius norm on $\herm(n, \hur)$, by Lemma \ref{cdklemma}, if $\hur$ is associative
\begin{align}
h(x\star x, z\star z) - h(x\star z, z \star x) = \tfrac{1}{4n}\tr \overline{[x, z]}^{t}[x, z] \leq \tfrac{n}{2}\left(|x|^{2}|z|^{2} - h(x, z)^{2}\right),
\end{align}
which shows the upper bound in \eqref{hermsect}. The characterization of the equality case in Lemma \ref{cdklemma} yields the characterization of the equality case in the upper bound of \eqref{hermsect}.
If $\hur = \cayley$, then by Lemma \ref{cdklemma},
\begin{align}
\tfrac{1}{4n}\tr \overline{[\la, y]}^{t}[\la, y] \leq \tfrac{n}{2}\left(|\la|^{2}|y|^{2} - h(\la, y)^{2}\right) =\tfrac{n}{2}\left(|x|^{2}|z|^{2} - h(x, z)^{2}\right) ,
\end{align}
(where $n  = 3$) and in \eqref{hermsect0b} this shows the upper bound in \eqref{hermsect}. The characterization of the equality case in Lemma \ref{cdklemma} again yields the characterization of the equality case in the upper bound of \eqref{hermsect}.
\end{proof}


\section{The Chern-do Carmo-Kobayashi inequality over real Hurwitz algebras}\label{bwsection}
The Frobenius bilinear form $f(x, y) = \re \tr \bar{x}^{t}y = \tfrac{1}{2}\tr(\bar{x}^{t}y + \bar{y}^{t}x)$ on $\mat(n, \hur)$ is symmetric \cite[proposition V.$2.1$]{Faraut-Koranyi} and positive definite (this is true for $\hur = \cayley$ even if $n > 3$).

\begin{lemma}\label{bwreductionlemma}
Let $\hur$ be a real Hurwitz algebra. Let $\mat(n, \hur)$ be metrized by $f(x, y) = \re \tr \bar{x}^{t}y$. The number
\begin{align}\label{bwdefined}
\bw = \bw(\mat(n, \hur)) = \sup\left\{\tfrac{|[x, y]|^{2}}{|x|^{2}|y|^{2}}: x, y \in \mat(n, \hur) \setminus\{0\}\right\}
\end{align}
is finite and
\begin{align}\label{bwpre}
&|[x, y]|^{2} \leq \bw(|x|^{2}|y|^{2} - f(x, y)^{2}) \leq \bw|x|^{2}|y|^{2},&& \text{for all}\,\, x, y \in \mat(n, \hur).
\end{align}
If $|[x, y]|^{2} \leq \bw(|x|^{2}|y|^{2} - f(x, y)^{2})$ then $p = |y|x - |x|y$ and $q = |y|x + |x|y$ satisfy $|[p, q]|^{2} = \bw|p|^{2}|q|^{2}$.
\end{lemma}
\begin{proof}
The Cauchy-Schwarz inequality implies directly $\bw(\mat(n, \hur)) \leq 4$, which shows $\bw(\mat(n, \hur))$ is finite but is not optimal. The following argument from the proof of \cite[Theorem $2.2$]{Wu-Liu} shows that the apparently weaker inequality $|[x, y]|^{2} \leq \bw(\mat(n, \hur))|x|^{2}|y|^{2}$ implies the first inequality of \eqref{bwpre}. If either $x$ or $y$ equals $0$ there is nothing to show, so suppose $x \neq 0$ and $y \neq 0$. Take $p = |y|x - |x|y$ and $q = |y|x + |x|y$. Then $[p, q] = 2|x||y|[x, y]$, and straightforward computation shows
\begin{align}\label{bwreduction}
\begin{split}
|[x, y]|^{2} = \tfrac{1}{4|x|^{2}|y|^{2}}|[p, q]|^{2} \leq \tfrac{1}{2|x|^{2}|y|^{2}}|p|^{2}|q|^{2} = 2(|x|^{2}|y|^{2} - f(x, y)^{2}).
\end{split}
\end{align}
If there is equality in the first inequality of \eqref{bwpre}, then \eqref{bwreduction} shows that $|[p, q]|^{2} = \bw|p|^{2}|q|^{2}$.
\end{proof}

The Böttcher-Wenzel inequality was first proved for real symmetric matrices, with a characterization of the equality $|[x, y]|^{2}_{f} = 2|x|_{f}^{2}|y|_{f}^{2}$, as \cite[Lemma $1$]{Chern-Docarmo-Kobayashi}. It was proved for $\mat(n, \com)$ in full generality by Z. Lu in \cite{Lu-normalscalar} and as stated in the proof of \cite[Theorem $3.1$]{Bottcher-Wenzel}. A nice exposition of the proof of \eqref{bw}, following \cite{Audenaert}, is in \cite[chapter $9$]{Zhan}. For variants and discussion see \cite{Audenaert, Bottcher-Wenzel-howbig, Bottcher-Wenzel, Cheng-Vong-Wenzel, Lu-normalscalar, Lu-Wenzel, Zhan} among many references. 

\begin{theorem}[\cite{Bottcher-Wenzel, Chern-Docarmo-Kobayashi, Lu-normalscalar}]\label{bwtheorem}
There holds $\bw(\mat(n, \com)) = 2$. That is, for the inner product $f(x, y) = \re \tr \bar{x}^{t}y$, 
\begin{align}\label{bw}
&|[x, y]|^{2} \leq 2(|x|^{2}|y|^{2} - f(x, y)^{2}) \leq 2|x|^{2}|y|^{2}, &&\text{for}\,\, x, y \in \mat(n, \com). 
\end{align}
\begin{enumerate}
\item\label{submax0} If either $x, y \in \herm(n, \com)$ or $x, y \in \su(n)$, then $|[x, y]|^{2} = 2|x|^{2}|y|^{2}$ if and only if $x$ and $y$ are simultaneously unitarily conjugate to scalar multiples of $e_{11} - e_{nn}$ and $e_{1n} + e_{n1}$.
\item\label{submax} If either $x, y \in \herm(n, \com)$ or $x, y \in \su(n)$, then $|[x, y]|^{2} = 2(|x|^{2}|y|^{2} - f(x, y)^{2})$ if and only if $x$ and $y$ either are linearly dependent over $\com$ or are simultaneously unitarily conjugate to scalar multiples of $e_{11} - e_{nn}$ and $e_{1n} + e_{n1}$.
\end{enumerate}
\end{theorem}

\begin{proof}
Characterizations of equality in \eqref{bw} are given in various references, for instance \cite{Bottcher-Wenzel}. Here these characterizations are needed when $x, y \in \herm(n, \com)$.
Since multiplication by $\j$ is an isometric real linear isomorphism from $\su(n)$ to $\herm(n, \com)$, the validity of either of \eqref{submax0} or \eqref{submax} for $x$ and $y$ both Hermitian or both anti-Hermitian implies its validity for the other class. The characterization of equality in \eqref{submax0} for Hermitian matrices is proved as it is in \cite{Chern-Docarmo-Kobayashi} for $\herm(n, \rea)$; see the proof of Lemma \ref{cdklemma} below. Suppose there holds $|[x, y]|^{2} = 2(|x|^{2}|y|^{2} - f(x, y)^{2})$ in \eqref{bw} for linearly independent $x, y \in \herm(n, \com)$. Then \eqref{bwreduction} shows that the Hermitian matrices $p$ and $q$ satisfy $|[p, q]|^{2} = 2|p|^{2}|q|^{2}$, so, by \eqref{submax0}, are simultaneously unitarily conjugate to (nonzero) scalar multiples of $e_{11} - e_{nn}$ and $e_{1n} + e_{n1}$. Since this implies that $|p|^{2}= |q|^{2}$, from $x = \tfrac{p+q}{2|x|}$ and $y = \tfrac{q - p}{2|y|}$ it follows that $f(x, y) = 0$, so, by \eqref{bwreduction}, there holds $|[x, y]|^{2} = 2|x|^{2}|y|^{2}$. By \eqref{submax0}, $x$ and $y$ are simultaneously unitarily conjugate to scalar multiples of $e_{11} - e_{nn}$ and $e_{1n} + e_{n1}$. This proves \eqref{submax}. 
\end{proof}

The inequality \eqref{bw} is not valid over $\quat$ and $\cayley$. Taking $x = \j I$ and $y = j I \in \mat(n, \quat)$ shows $\bw(\mat(n, \hur)) \geq 4$ when $\hur \in \{\quat, \cayley\}$. In \cite[Theorem $3.1$]{Ge-Li-Zhou} it is shown that $\bw(\mat(n, \quat)) = 4$ and the equality case is characterized. Nonetheless, the proof of \eqref{bw} given in \cite{Chern-Docarmo-Kobayashi} for $\herm(n, \rea)$ adapts for $\herm(n, \hur)$ over any real Hurwitz algebra $\hur \in \{\rea, \com, \quat, \cayley\}$. Over $\rea$ or $\com$, because a Hermitian matrix is unitarily diagonalizable, it suffices to prove the claim in the case one of the matrices is diagonal. This proof requires the prinicipal axis theorem, which is valid over $\quat$ and $\cayley$, and one has to check that the noncommutativity of $\quat$ and $\cayley$ does not effect the rest of the argument. This works over $\quat$, but the reduction via diagonalization faces an obstacle over $\cayley$, namely that, for $x, y \in \herm(n, \hur)$, the inequality \eqref{bw} is not manifestly invariant with respect to the action of $\Aut(\herm(n, \hur), \star)$. Over associative $\hur$, $\Aut(\herm(n, \hur))$ preserves the commutator of the matrix product on $\mat(n, \hur)$ restricted to $\herm(n, \hur)$, but when $\hur = \cayley$, in which case $\Aut(\herm(3, \cayley))$ is the compact form of the simple real Lie group of type $F_{4}$ \cite{Schafer-book}, it is not clear to the author whether this is true. However the partial result, that \eqref{bw} holds when $x$ is diagonal, is true, and this suffices for the application to sectional nonassociativity in the proof of Theorem \ref{hermsecttheorem}.

\begin{lemma}\label{cdklemma}
Let $\hur$ be a real Hurwitz algebra. 
Let $f(x, y) = \re \tr \bar{x}^{t}y$ on $\mat(n, \hur)$.
\begin{enumerate}
\item If $\hur$ is associative, for all $x, y \in \herm(n, \hur)$ there holds 
\begin{align}\label{cdk}
|[x, y]|^{2} \leq 2(|x|^{2}|y|^{2} - f(x, y)^{2})\leq 2|x|^{2}|y|^{2}.
\end{align}
\item If $\hur = \cayley$ and $n = 3$, \eqref{cdk} holds for all $y \in \herm(3, \cayley)$ and all \emph{diagonal} $x \in \herm(3, \cayley)$.
\item\label{maximalcommutator} The equality $|[x, y]|^{2} = 2|x|^{2}|y|^{2}$ holds in \eqref{cdk} (where, if $\hur = \cayley$ and $n = 3$, $x$ is assumed diagonal) if and only if $x$ and $y$ are simultaneously equivalent via an automorphism of the Jordan algebra $(\herm(n, \hur), \star)$ to real multiples of $e_{11} - e_{nn}$ and $e_{1n} + e_{n1}$.
\item\label{submaximalcommutator} The equality $|[x, y]|^{2} = 2(|x|^{2}|y|^{2} - f(x, y)^{2})$ holds in \eqref{cdk} (where, if $\hur = \cayley$ and $n = 3$, $x$ is assumed diagonal) if and only if $x$ and $y$ are either linearly dependent over $\hur$ or are simultaneously equivalent via an automorphism of the Jordan algebra $(\herm(n, \hur), \star)$ to real multiples of $e_{11} - e_{nn}$ and $e_{1n} + e_{n1}$. 
\end{enumerate}
\end{lemma}
\begin{proof}
Suppose $x \in \herm(n, \hur)$ is diagonal, where $n = 3$ if $\hur = \cayley$. Since $x$ is Hermitian, its entries are real. Let $x_{i} = x_{ii}$. Then $[x, y]_{ij} = x_{i}y_{ij} - y_{ij}x_{j} = (x_{i} - x_{j})y_{ij}$, the last equality because $x_{j}$ is real, so in the center of $\hur$. Now the proof goes through as in \cite{Chern-Docarmo-Kobayashi}:
\begin{align}\label{cdk0}
\begin{split}
|[x, y]|^{2} &= \sum_{i, j = 1}^{n}(x_{i} -x_{j})^{2}|y_{ij}|^{2}  \leq 2\sum_{i, j = 1}^{n}(x_{i}^{2} + x_{j}^{2})|y_{ij}|^{2} \leq 2\sum_{k = 1}^{n}x_{k}^{2}\sum_{i, j = 1}^{n}|y_{ij}|^{2} = 2|x|^{2}|y|^{2}.
\end{split}
\end{align}
This proves $|[x, y]|^{2} \leq 2|x|^{2}|y|^{2}$ for diagonal $x$. By the principal axis theorem, any element of $(\herm(n, \hur), \star)$ is equivalent via an automorphism of $(\herm(n, \hur), \star)$ to a diagonal matrix. 
Since an automorphism of $(\herm(n, \hur), \star)$ preserves the trace, it is isometric, and, when $\hur$ is associative it preserves the commutator of the matrix product on $\mat(n, \hur)$ restricted to $\herm(n, \hur)$, when $\hur$ is associative to prove $|[x, y]|^{2} \leq 2|x|^{2}|y|^{2}$ in general it suffices to prove it when $x$ is diagonal. By Lemma \ref{bwreductionlemma}, in either case $|[x, y]|^{2} \leq 2|x|^{2}|y|^{2}$ implies \eqref{cdk}.

The characterization of the equality case \eqref{maximalcommutator} goes through as in the proof of \cite[Lemma $1$]{Chern-Docarmo-Kobayashi}. Precisely, if there is equality in \eqref{cdk0}, then 
\begin{align}
0 = 4\sum_{i = 1}^{n}x_{i}^{2}|y_{ii}|^{2} + 2\sum_{i\neq j}(x_{i} + x_{j})^{2}|y_{ij}|^{2} =  2\sum_{k = 1}^{n}x_{k}^{2}\sum_{i = 1}^{n}|y_{ii}|^{2} + 2\sum_{i\neq j}(x_{i} + x_{j})^{2}|y_{ij}|^{2}.
\end{align}
If $x$ is nonzero, this implies $y_{ii} = 0$ for $1 \leq i \leq n$, $x_{i} + x_{j} = 0$ when $y_{ij} \neq 0$ for $i \neq j$, and that at exactly two of $x_{1}, \dots, x_{n}$ are nonzero. Suppose $i \neq j$ are the indices such that $x_{i} = -x_{j} \neq 0$. It follows that $y_{kl} = 0$ if $\{k, l\} \neq \{i, j\}$. Hence $x$ and $y$ are scalar multiples of $e_{ii} - e_{jj}$ and $e_{ij} + e_{ji}$, respectively. Since $\Aut(\herm(n, \hur))$ contains a subgroup acting as permutations on the diagonal subalgebra of $\herm(n, \hur)$, it can be assumed that $i = 1$ and $j = n$. The proof of the equality case \eqref{submaximalcommutator} follows from \eqref{maximalcommutator} exactly as in the proof of \eqref{submax} of Theorem \ref{bwtheorem}.
\end{proof}
 
\begin{remark}
Although the same difficulty related to diagonalization occurs in the proof of Theorem \ref{hermsecttheorem} as occurs in the proof of Lemma \ref{cdklemma}, the end result in Theorem \ref{hermsecttheorem} is stated in terms manifestly invariant with respect to $\Aut(\herm(n, \hur))$, and as a result takes the same form whether or not $\hur$ is associative. This observation suggests that the upper bound on sectional nonassociativity, which is closely related to the commutator bound, is the more natural bound to consider, at least from the algebraic point of view, although it could also simply reflect a technical deficiency in the proof of Lemma \ref{cdklemma}.
\end{remark}

\begin{remark}
It would be useful to know whether \eqref{cdk} is true for $\herm(3, \cayley)$. This seems likely. More, generally, it would be interesting to know whether \eqref{cdk} is true for $\herm(n, \cayley)$ for $n > 3$. This would follow from a principal axis theorem and a characterization of the automorphism group, but it appears that neither is known. See \cite{Zohrabi-Zumanovich} for what is known in this regard.
\end{remark}

\begin{remark}\label{killingbwremark}
For a compact semisimple real Lie algebra $\g$ with Killing form $B_{\g}$ the number
\begin{align}\label{bwgdefined}
\bw(\g) =\bwu(\g, [\dum, \dum], -B_{\g}) = \sup_{x, y \in \g: x \wedge y \neq 0}\frac{-B_{\g}([x, y], [x, y])}{B_{\g}(x, x)B_{\g}(y, y) - B_{g}(x, y)^{2}}
\end{align}
is positive and finite, by the Cauchy-Schwarz inequality. Its value is a basic automorphism invariant of $(\g, [\dum, \dum])$. Its calculation for particular $\g$ can be viewed as a refinement of the Böttcher-Wenzel inequality. On the other hand, for some compact simple real Lie algebras $\bw(\g)$ can be estimated using the Böttcher-Wenzel inequality. 
For $\su(n)$ and $\so(n)$, $B_{\su(n)}(x, y) =-2nf(x, y)$ and $B_{\so(n)}(x, y) = -(n-2)f(x, y)$ where $f(x, y) = \tr \bar{x}^{t}y$, so in these cases it follows from the Böttcher-Wenzel inequality, Theorem \ref{bwtheorem}, that
\begin{align}
&\bw(\so(n)) \leq \tfrac{2}{n-2},& &\bw(\su(n)) \leq \tfrac{1}{n}.
\end{align}
For the special case of $\so(n)$, viewed as antisymmetric matrices, equipped with the Frobenius norm $f$, the essentially equivalent quantity $ \sup_{x, y \in \g: x \wedge y \neq 0}\tfrac{|[x, y]|_{f}}{|x|_{f}|y|_{f}}$ was shown to equal $\sqrt{2}$ if $n \geq 4$ in \cite[Theorem $6$]{Bloch-Iserles}. By Lemma \ref{ghboundlemma}, these estimates yield sharp numerical bounds on the sectional nonassociativites of the subspaces spanned by decomposable elements in algebras such as $\so(m)\tensor \so(n)$, $\so(m) \tensor \su(n)$, and $\su(m)\tensor \su(n)$.
\end{remark}

Lemma \ref{ghboundlemma} shows a partial lower bound on the sectional nonassociativities of the tensor product of Lie algebras.
\begin{lemma}\label{ghboundlemma}
For compact semisimple real Lie algebras $\g$ and $\h$, let $\tau_{\mlt}$ be the Killing form of the tensor product algebra  $(\g \tensor \h, \mlt)$, defined by $\tau_{\mlt}(a, b) = \tr L_{\mlt}(a)L_{\mlt}(b)$ for $a, b \in \g \tensor \h$. The sectional nonassociativity of the subspace of $(\g \tensor \h, \mlt, \tau_{\mlt})$ spanned by decomposable elements $a_{1}\tensor b_{1}, a_{2}\tensor b_{2} \in \g \tensor \h$ satisfies
\begin{align}\label{sectghbounds}
0 \geq \sect(a_{1}\tensor b_{1}, a_{2}\tensor b_{2}) \geq - \bw(\g)\bw(\h).
\end{align}
\end{lemma}

\begin{proof}
Because $\g$ and $\h$ are compact, their Killing forms, $B_{\g}$ and $B_{\h}$, satisfy $-B_{\g}([a_{1}, a_{2}], [a_{1}, a_{2}]) \geq 0$ and $-B_{\h}([b_{1}, b_{2}], [b_{1}, b_{2}])\geq 0$.
The upper bound in \eqref{sectghbounds} follows from 
\begin{align}\label{sectgh1}
\begin{split}
\sect(a_{1}\tensor b_{1}, a_{2}\tensor b_{2}) & = -\tfrac{B_{\g}([a_{1}, a_{2}], [a_{1}, a_{2}])B_{\h}([b_{1}, b_{2}], [b_{1}, b_{2}])}{B_{\g}(a_{1}, a_{1})
B_{\g}(a_{2}, a_{2})B_{\h}(b_{1}, b_{1})
B_{\h}(b_{2}, b_{2}) - B_{\g}(a_{1}, a_{2})^{2}B_{\h}(b_{1}, b_{2})^{2}}\leq 0,
\end{split}
\end{align}
where the positivity of the denominator follows from the Cauchy-Schwarz inequality. On the other hand, by definition of $\bw(\g)$ and $\bw(\h)$,
\begin{align}
\begin{split}
-&\tfrac{\sect(a_{1}\tensor b_{1}, a_{2}\tensor b_{2}) }{\bw(\g)\bw(\h)} = \tfrac{1}{\bw(\g)\bw(\h)}\tfrac{\left(-B_{\g}([a_{1}, a_{2}], [a_{1}, a_{2}])\right)\left(-B_{\h}([b_{1}, b_{2}], [b_{1}, b_{2}])\right)}{B_{\g}(a_{1}, a_{1})
B_{\g}(a_{2}, a_{2})B_{\h}(b_{1}, b_{1})B_{\h}(b_{2}, b_{2}) - B_{\g}(a_{1}, a_{2})^{2}B_{\h}(b_{1}, b_{2})^{2}}\\
& \leq \tfrac{\left(B_{\g}(a_{1}, a_{1})B_{\g}(a_{2}, a_{2}) - B_{\g}(a_{1}, a_{2})^{2} \right)\left(B_{\h}(b_{1}, b_{1})B_{\h}(b_{2}, b_{2}) - B_{\h}(b_{1}, b_{2})^{2} \right)}{B_{\g}(a_{1}, a_{1})B_{\g}(a_{2}, a_{2})B_{\h}(b_{1}, b_{1})B_{\h}(b_{2}, b_{2}) - B_{\g}(a_{1}, a_{2})^{2}B_{\h}(b_{1}, b_{2})^{2}}\\
& =1 - \tfrac{B_{\g}(a_{1}, a_{2})^{2}\left(B_{\h}(b_{1}, b_{1})B_{\h}(b_{2}, b_{2}) - B_{\h}(b_{1}, b_{2})^{2} \right) + B_{\h}(b_{1}, b_{2})^{2}\left(B_{\g}(a_{1}, a_{1})B_{\g}(a_{2}, a_{2}) - B_{\g}(a_{1}, a_{2})^{2} \right)}{B_{\g}(a_{1}, a_{1})B_{\g}(a_{2}, a_{2})B_{\h}(b_{1}, b_{1})B_{\h}(b_{2}, b_{2}) - B_{\g}(a_{1}, a_{2})^{2}B_{\h}(b_{1}, b_{2})^{2}}\leq 1,
\end{split}
\end{align}
which shows the lower bound in \eqref{sectghbounds}.
\end{proof} 
Note that Lemma \ref{ghboundlemma} does not imply that $\g \tensor \h$ has nonpositive sectional nonassociativity. In fact, for any compact simple real Lie algebra $\g$, $\so(3)\tensor\g$ has sectional nonassociativities of both signs \cite{Fox-simplicial}. 

\begin{remark}\label{vinbergremark}
Although this does not seem to be realized widely, for anti-Hermitian matrices the inequality \eqref{bw} is closely related to Vinberg's results on invariant norms on compact simple Lie algebras in \cite{Vinberg-invariantnorms} applied in the special case of $\su(n)$. Precisely, Vinberg shows that, for an invariant norm, $||\dum||$ on a compact simple Lie algebra $\g$ the quantity $\theta(x) = \sup_{0 \neq y \in \g}\tfrac{||[x, y]||}{||y||}$ does not depend on the choice of invariant norm, and equals the spectral norm of $\ad_{\g}(x)$. Since, by definition, $\theta([x, y]) \leq \theta(x)\theta(y)$, taking $\g$ to be a compact simple Lie algebra of matrices, e.g. antisymmetric or anti-Hermitian matrices, and taking $||\dum ||$ to be the Frobenius norm there results $|[x, y]|_{f} \leq \theta(x)|y|_{f}$. It suffices then to check that the spectral norm of $\ad_{\g}(x)$ is no greater than $\sqrt{2}|x|_{f}$. In the case $\g = \su(n)$ this can be shown as follows. Since any anti-Hermitian matrix is unitarily conjugate to a diagonal matrix and the spectral norms of the adjoint representations of unitarily conjugate matrices are the same, it suffices to consider the case of diagonal $x \in \su(n)$. In this case the nonzero eigenvalues of $\ad_{\su(n)}(x)$ have the form $\la_{i} - \la_{j}$ where $\la_{1}, \dots, \la_{n}$ are the diagonal elements of $x$. Since $|\la_{i} - \la_{j}|^{2} \leq 2|\la|^{2} \leq 2|x|^{2}_{f}$, where $\la$ is the element of $x$ with the greatest modulus, the spectral norm of $\ad_{\su(n)}(x)$ is no greater than $\sqrt{2}|x|_{f}$. 
\end{remark}

\section{Consequences for commutative algebras of conditions on sectional nonassociativity}\label{consequencessection}
This section describes some general structural consequences for commutative algebras of assumptions on the sectional nonassociativity.

Consequences of nonnegative sectional nonassociativity for commutative algebras can be found in \cite[section $17$]{Conway-monster} and in \cite{Meyer-Neutsch}, where they are applied to the study of maximal associative subalgebras of the Griess algebra. By the analogy advocated here, associative subalgebras are analogous to flat submanifolds of a nonnegatively curved Riemannian manifold. Here such results are cast in a somewhat more general setting. 

An element $e \in \alg$ is \emph{idempotent} if $e \mlt e = e$ and \emph{square-zero} if $e \mlt e = 0$. 
Let $\idem(\alg, \mlt) = \{0 \neq e \in \alg: e \mlt e = e\}$ and $\szero(\alg, \mlt) = \{0 \neq z \in \alg: z \mlt z  = 0\}$.
\begin{lemma}\label{complexidempotentlemma}
Let $(\alg, \mlt, h)$ be a Euclidean metrized commutative algebra.
\begin{enumerate}
\item\label{complexzeroclaim} If $a + \j b \in \szero(\alg\tensor_{\rea}\com, \mlt)$ and the $\rea$-span of $a$ and $b$ is two-dimensional, then
\begin{align}\label{complexszerosect}
\sect(a, b) = \tfrac{|a \mlt a|^{2}}{|a\wedge  b|^{2}} \geq 0,
\end{align}
with equality if and only if $a$ and $b$ span a trivial subalgebra of $(\alg, \mlt)$.
\item If $a + \j b \in \idem(\alg\tensor_{\rea}\com, \mlt)$ and the $\rea$-span of $a$ and $b$ is two-dimensional, then
\begin{align}\label{complexidemsect}
\sect(a, b) = \tfrac{1}{4}\tfrac{4|b \mlt b|^{2} + |b|^{2}}{|a\wedge  b|^{2}} > 0.
\end{align}
\end{enumerate}
\end{lemma}
\begin{proof}
That $a + \j b \in \szero(\alg\tensor_{\rea}\com, \mlt)$ is equivalent to the equations $a\mlt a = b \mlt b$ and $a \mlt b  = 0$, and in \eqref{sectnadefined} these yield \eqref{complexszerosect}. If equality holds in \eqref{complexszerosect}, then $b\mlt b = a \mlt a = 0$, so $a$ and $b$ generate a trivial subalgebra. 
Suppose $a + \j b \in \idem(\alg\tensor_{\rea}\com, \mlt)$, so that $a\mlt a - b \mlt b = a$ and $2 a \mlt b = b$. Then 
\begin{align}
\begin{split}
h(a\mlt a, b\mlt b) & - h(a\mlt b, a \mlt b) = h(b \mlt b + a, b\mlt b) = |b \mlt b|^{2} + h(b, a\mlt b) - \tfrac{1}{4}|b|^{2} = |b\mlt b|^{2} + \tfrac{1}{4}|b|^{2},
\end{split}
\end{align}
and in \eqref{sectnadefined} this yields \eqref{complexidemsect}. In \eqref{complexidemsect}, equality cannot hold because it would imply $b = 0$, contrary to hypothesis.
\end{proof}	

\begin{theorem}\label{finiteautotheorem}
Let $(\alg, \mlt, h)$ be a Euclidean metrized commutative algebra.
\begin{enumerate}
\item If $(\alg, \mlt, h)$ has negative sectional nonassociativity, then $\Aut(\alg, \mlt, h)$ is finite.
\item If $(\alg, \mlt, h)$ has nonpositive sectional nonassociativity, either $\Aut(\alg, \mlt, h)$ is finite or $(\alg, \mlt, h)$ contains a trivial subalgebra of dimension $1$ or $2$ (the possibilities are not exclusive).
\end{enumerate}
\end{theorem}
\begin{proof}
For $k \geq 3$, let $\theta$ be a multilinear $k$-form on a complex vector space $\ste$. By a theorem of H. Suzuki \cite[Theorem B]{Suzuki-automorphismgroups}, either the group of linear automorphisms of $\theta$ is finite or there is a nonzero vector $v\in \ste$ such that $\theta(v, \dots, v, w) = 0$ for all $w\in \ste$. Regard $\mu(x, y, z) = h(x\mlt y, z)$ as a trilinear form on $\alg \tensor_{\rea} \com$. An orthogonal automorphism of $(\alg, \mlt, h)$ preserves $\mu$, and extends to an automorphism of $\alg \tensor_{\rea} \com$. By Suzuki's theorem if $\Aut(\alg, \mlt, h)$ is not finite there exists $a + \j b \in \alg \tensor_{\rea} \com$ such that $h((a + \j b)\mlt(a + \j b), w) = 0$ for all $w \in \alg \tensor_{\rea} \com$. By the nondegeneracy of $h$, $a + \j b \in \szero(\alg \tensor_{\rea} \com, \mlt)$. Note that it is not asserted that $a$ and $b$ are linearly independent over $\rea$. The conclusion follows from \eqref{complexzeroclaim} of Lemma \ref{complexidempotentlemma}. 
\end{proof}

For a Euclidean metrized commutative algebra $(\alg, \mlt, h)$, the endomorphism $L_{\mlt}(e)$ is $h$-self-adjoint for any $e \in \alg$. If $e \in \idem(\alg, \mlt)$, then $L_{\mlt}(e)$ preserves $\eperp =  \{y \in \alg: h(e, y) =0\}$, for if $h(e, y) = 0$ then $h(L_{\mlt}(e)y, e) = h(y, e \mlt e) = h(y, e) = 0$. Define the \emph{orthogonal spectrum}
\begin{align}
\specp(e) = \{\la \in \fie: \text{there is}\,\, x \in \fie\{ e \}^{\perp}\,\, \text{such that}\,\, L_{\mlt}(e)x = \la x\},
\end{align} 
so that $\spec(e) = \specp(e) \cup \{1\}$. Because $L_{\mlt}(e)$ preserves $\eperp$ it has an eigenvector in $\eperp$. Because $h(e, e) > 0$, this eigenvector is not a multiple of $e$, so in this case $\specp(e)$ is not empty.

\begin{lemma}\label{eigensectlemma}
Let $(\alg, \mlt, h)$ be a Euclidean metrized commutative algebra of dimension $n \geq 2$. If $0 \neq e \in \idem(\alg,\mlt)$, then $4\sect(e, x) \leq |e|^{-2}$ for all $x \in \alg$ such that $x \wedge e\neq 0$, with equality if and only if $x - h(e, e)^{-1}h(e, x)e \in \ker(L_{\mlt}(e) - \tfrac{1}{2}\Id_{\alg})$. 
\end{lemma}
\begin{proof}
Let $e \in \idem(\alg,\mlt)$. 
Because $\sect(e, x)$ depends only on the subspace spanned by $e$ and $x$, in computing $\sect(e, x)$ it may be supposed that $h(e, x) = 0$. 
Write $x = \sum_{i = 1}^{n-1}x_{i}$ where $x_{1}, \dots, x_{n-1}\in \eperp$ are $h$-orthogonal eigenvectors of $L_{\mlt}(e)$ with (not necessarily distinct) eigenvalues $\la_{1}, \dots, \la_{n-1}$. There results
\begin{align}\label{sectexest}
\begin{split}
\sect(e, x) &= \frac{h(e\mlt e, x \mlt x) - h(e\mlt x, e \mlt x)}{|e|^{2}|x|^{2}} = \frac{h(e\mlt x, x) - h(e\mlt x, e \mlt x)}{|e|^{2}|x|^{2}} = |e|^{-2}\sum_{i = 1}^{n-1}\la_{i}(1 - \la_{i})\frac{|x_{i}|^{2}}{|x|^{2}}.
\end{split}
\end{align}
Because $\sum_{i = 1}^{n-1}|x_{i}|^{2} = |x|^{2}$, \eqref{sectexest} exhibits $\sect(e, x)$ as a convex combination, so implies
\begin{align}\label{sectexest2}
\begin{split}
 |e|^{-2}\min\{\la_{i}(1 - \la_{i}):1 \leq i \leq n-1\} \leq \sect(e, x)  \leq  |e|^{-2}\max\{\la_{i}(1 - \la_{i}):1 \leq i \leq n-1\}.
\end{split}
\end{align}
As $\la(1 - \la) \leq 1/4$ for all $\la \in \rea$,  \eqref{sectexest2} yields $4\sect(e, x) \leq |e|^{-2}$. There holds equality if and only if $\la_{i} = 1/2$ for $1 \leq i \leq n-1$, in which case $x = \sum_{i = 1}^{n-1}x_{i}$ is an eigenvector of $L_{\mlt}(e)$ with eigenvalue $1/2$.
\end{proof}
\begin{lemma}\label{nonpossquareslemma}
Let $(\alg, \mlt, h)$ be a Euclidean metrized commutative algebra of dimension $n \geq 2$. Let $0 \neq e \in \idem(\alg, \mlt)$. 
\begin{enumerate}
\item \label{npb1} If there is $b \in \rea$ such that $\sect(\suba) \leq b$ for all two-dimensional subspaces $\suba \subset \alg$ containing $e$, then $\specp(e) \subset (-\infty, \tfrac{1 - \sqrt{B}}{2}] \cup [\tfrac{1+ \sqrt{B}}{2}, \infty)$, where $B = \max\{1 - 4b|e|^{2}, 0\}$. In particular, if $(\alg, \mlt, h)$ has nonpositive sectional nonassociativity, then $\specp(e) \cap (0, 1)  = \emptyset$. 
\item \label{npb2} If there is $b \in \rea$ such that $\sect(\suba) \geq b$ for all two-dimensional subspaces $\suba \subset \alg$ containing $e$, then $b \leq 1/(4|e|^{2})$ and $\specp(e)\subset [\tfrac{1 - \sqrt{B}}{2}, \tfrac{1+ \sqrt{B}}{2}]$ for $B = 1 - 4b|e|^{2} \geq 0$. In particular, if $(\alg, \mlt, h)$ has nonnegative sectional nonassociativity, then $\specp(e) \subset [0, 1]$.
\end{enumerate}
\end{lemma}

\begin{proof}
In the setting of \eqref{npb1}, if $0 \neq y \in \alg$ is an eigenvector of $L_{\mlt}(e)$ with eigenvalue $\la$ and orthogonal to $e$, then, by \eqref{sectexest}, $\la(1 -\la) = |e|^{2}\sect(e, y) \leq b|e|^{2}$, so $\la^{2} - \la + b|e|^{2} \geq 0$, which forces $\la$ to be in the indicated range. In the setting of \eqref{npb2}, the same argument shows that $\la^{2} - \la + b|e|^{2} \geq 0$, which forces $1 - 4b|e|^{2} \geq 0$ and forces $\la$ to be in the indicated range.
\end{proof}

Next there are given some results showing that nonnegative sectional nonassociativity precludes the existence of square-zero elements.

\begin{lemma}\label{faithfulmultiplicationlemma}
If a nontrivial metrized algebra $(\alg, \mlt, h)$ satisfies $\alg \mlt \alg = \alg$ then its multiplication is faithful. In particular, the multiplication of a metrized semisimple commutative algebra $(\alg, \mlt, h)$ is faithful.
\end{lemma}
\begin{proof}
If $z \in \ker L_{\mlt}$, then, by the invariance of $h$, $0 = h(L_{\mlt}(z)x, y) = h(z, x\mlt y)$ for all $x, y \in \alg$. This shows $\ker L$ is contained in the orthocomplement $(\alg \mlt \alg)^{\perp}$ of $\alg \mlt \alg$. In particular, if $\alg \mlt \alg = \alg$, then $\ker L_{\mlt} \subset (\alg \mlt \alg)^{\perp} = \alg^{\perp} = \{0\}$, so $L_{\mlt}$ is injective. 
If $\alg = \oplus_{i =1 }^{k}\alg_{i}$ is a direct sum of simple ideals, then $\alg_{i}\mlt \alg_{j} = \{0\}$ if $i \neq j$, and, since $\alg_{i}$ is simple, $\alg_{i}\mlt \alg_{i} = \alg_{i}$, so $\alg \mlt \alg = \sum_{i}\alg_{i}\mlt \alg_{i} = \sum_{i}\alg_{i} = \alg$, and $L_{\mlt}$ is injective. By the invariance of $h$, $h(R_{\mlt}(y)x, z) = h(x, L_{\mlt}(y)z)$ for all $x, y, z \in \alg$, so were $R_{\mlt}$ not injective, $L_{\mlt}$ would not be injective. 
\end{proof}

\begin{lemma}\label{rchassoclemma}
Let $(\alg,\mlt, h)$ be a Euclidean metrized commutative algebra. Let $0 \neq z \in \szero(\alg, \mlt)$.
 For $0 \neq y \in \alg$ not in the span of $z$, $\sect(z, y) \leq 0$, with equality if and only if $z \mlt y = 0$. In particular, $\szero(\alg, \mlt) = \{0\}$ if there holds either of the following conditions:
\begin{enumerate}
\item $(\alg, \mlt, h)$ has positive sectional nonassociativity.
\item\label{nsz2} $(\alg, \mlt, h)$ has nonnegative sectional nonassociativity and is semisimple.
\end{enumerate}
\end{lemma}

\begin{proof}
Let $0 \neq z \in \szero(\alg, \mlt)$. For $y \in \alg$ linearly independent of $z$, there holds $\sect(z, y)|z\wedge y|^{2} = - |L_{\mlt}(z)y|^{2} \leq 0$, with equality if and only if $z \mlt y = 0$. In particular, $(\alg, \mlt,h)$ cannot have positive sectional nonassociativity, and if $(\alg, \mlt,h)$ has nonnegative sectional nonassociativity, then $L_{\mlt}(z)y = 0$ for all $y \in \alg$, so $L_{\mlt}(z) =0$. 
By Lemma \ref{faithfulmultiplicationlemma}, if $L_{\mlt}(z) = 0$, then $(\alg, \mlt)$ is not semisimple. 
\end{proof}

A commutative algebra is \emph{exact} if $\tr L_{\mlt}(x) = 0$ for all $x \in \alg$. Exactness is a condition analogous to unimodularity of a Lie algebra. Evidently, an exact commutative algebra is not unital. 

\begin{lemma}
A Euclidean metrized commutative algebra $(\alg, \mlt, h)$ with nontrivial multiplication and nonnegative sectional nonassociativity is not exact.
\end{lemma}

\begin{proof}
By \cite[Lemma $2.3$]{Tkachev-correction}, there exists a nontrivial idempotent $e \in \idem(\alg, \mlt)$. Because $L_{\mlt}(e)$ is self-adjoint, there is an $h$-orthogonal basis $\{e, z_{1}, \dots, z_{n-1}\}$ of $\alg$ such that $z_{1}, \dots, z_{n-1}$ are eigenvectors of $L_{\mlt}(e)$ with (not necessarily distinct) eigenvalues $\la_{1}, \dots, \la_{n-1} \in \rea$. By \eqref{npb2} of Lemma \ref{nonpossquareslemma}, nonnegativity of the sectional nonassociativity implies $\la_{i} \in [0, 1]$ for $1 \leq i \leq n-1$. It follows that $\tr L_{\mlt}(e) = 1 + \sum_{i = 1}^{n-1}\la_{i} \geq 1$, so $(\alg, \mlt)$ is not exact.
\end{proof}

In a commutative $\rea$-algebra $(\alg, \mlt)$ the set $\squares(\alg, \mlt) = \{x \mlt x: x \in \alg\}$ is a closed cone, the \emph{cone of squares}. By definition, $\idem(\alg, \mlt) \subset \squares(\alg, \mlt)$.

\begin{lemma}\label{squarelemma}
Let $(\alg, \mlt, h)$ be a Euclidean metrized commutative algebra having nonnegative sectional nonassociativity. The set
\begin{align}
\cone(\alg, \mlt, h) = \{x \in \alg: L_{\mlt}(x) \,\,\text{is nonnegative}\}
\end{align}
is a closed convex cone containing $\squares(\alg, \mlt)$ for all $x \in \alg$ and consequently containing $\idem(\alg, \mlt)$.
\end{lemma}

\begin{proof}
That the sectional nonassociativity is nonnegative means that $L_{\mlt}(x\mlt x) \geq L_{\mlt}(x)^{2} \geq 0$ for all $x \in \alg$. This shows $\idem(\alg, \mlt)\subset \squares(\alg, \mlt, h) \subset \cone(\alg, \mlt, g)$. In a Euclidean metrized commutative algebra every multiplication endomorphism $L_{\mlt}(x)$ is self-adjoint, so diagonalizable with real eigenvalues. If $x, y \in \alg$ and $t \in [0, 1]$ then $L_{\mlt}(tx + (1-t)y) = tL_{\mlt}(x) + (1-t)L_{\mlt}(y) \geq 0$ because the set of nonnegative definite self-adjoint endomorphisms of a Euclidean vector space is a closed convex cone.
\end{proof}

With the hypotheses of Lemma \ref{squarelemma}, are the extreme rays of $\squares(\alg, \mlt)$ generated by elements of $\idem(\alg, \mlt)$? This is true for the cones of squares in the simple Euclidean Jordan algebras over $\rea$ or $\com$; these are the cones of nonnegative definite Hermitian matrices and their extremal rays are generated by the rank one idempotents. For the Griess algebras of OZ VOAs (Example \ref{voaexample}) it seems the question has not been studied.

\bibliographystyle{amsplain}

\def\polhk#1{\setbox0=\hbox{#1}{\ooalign{\hidewidth
  \lower1.5ex\hbox{`}\hidewidth\crcr\unhbox0}}} \def\cprime{$'$}
  \def\cprime{$'$} \def\cprime{$'$}
  \def\polhk#1{\setbox0=\hbox{#1}{\ooalign{\hidewidth
  \lower1.5ex\hbox{`}\hidewidth\crcr\unhbox0}}} \def\cprime{$'$}
  \def\cprime{$'$} \def\cprime{$'$} \def\cprime{$'$}
  \def\polhk#1{\setbox0=\hbox{#1}{\ooalign{\hidewidth
  \lower1.5ex\hbox{`}\hidewidth\crcr\unhbox0}}} \def\cprime{$'$}
  \def\Dbar{\leavevmode\lower.6ex\hbox to 0pt{\hskip-.23ex \accent"16\hss}D}
  \def\cprime{$'$} \def\cprime{$'$} \def\cprime{$'$} \def\cprime{$'$}
  \def\cprime{$'$} \def\cprime{$'$} \def\cprime{$'$} \def\cprime{$'$}
  \def\cprime{$'$} \def\cprime{$'$} \def\cprime{$'$} \def\dbar{\leavevmode\hbox
  to 0pt{\hskip.2ex \accent"16\hss}d} \def\cprime{$'$} \def\cprime{$'$}
  \def\cprime{$'$} \def\cprime{$'$} \def\cprime{$'$} \def\cprime{$'$}
  \def\cprime{$'$} \def\cprime{$'$} \def\cprime{$'$} \def\cprime{$'$}
  \def\cprime{$'$} \def\cprime{$'$} \def\cprime{$'$} \def\cprime{$'$}
  \def\cprime{$'$} \def\cprime{$'$} \def\cprime{$'$} \def\cprime{$'$}
  \def\cprime{$'$} \def\cprime{$'$} \def\cprime{$'$} \def\cprime{$'$}
  \def\cprime{$'$} \def\cprime{$'$} \def\cprime{$'$} \def\cprime{$'$}
  \def\cprime{$'$} \def\cprime{$'$} \def\cprime{$'$} \def\cprime{$'$}
  \def\cprime{$'$} \def\cprime{$'$} \def\cprime{$'$}
\providecommand{\bysame}{\leavevmode\hbox to3em{\hrulefill}\thinspace}
\providecommand{\MR}{\relax\ifhmode\unskip\space\fi MR }
\providecommand{\MRhref}[2]{%
  \href{http://www.ams.org/mathscinet-getitem?mr=#1}{#2}
}
\providecommand{\href}[2]{#2}

\end{document}